\documentclass[11pt]{article}

\title{Large cliques in sparse random intersection graphs
\footnote{Supported by the 
Research Council of Lithuania (MIP-052/2010, MIP-067/2013).} }
\date{2013-09-29}

\usepackage{amsmath}
\usepackage{amssymb}
\usepackage{latexsym}

\usepackage[utf8]{inputenc}

\usepackage{enumerate}
\usepackage{hyperref}
\hypersetup{
    colorlinks=false,
    pdfborder={0 0 0},
}

\newenvironment{proof}{\noindent{\bf Proof}}{\hspace*{\fill}$\Box$}
\newenvironment{proofof}[1]{%
\noindent {\bf Proof of #1}}%
{\hspace*{\fill}$\Box$}

\newtheorem{theorem}{Theorem}[section]
\newtheorem{lemma} [theorem] {Lemma}
\newtheorem{prop} [theorem] {Proposition}
\newtheorem{remark} [theorem] {Remark}

\def\E{{\mathbb E}\,}

\let\eps=\epsilon
\def\enddiscard{}
\long\def\discard#1\enddiscard{}

\newcommand{\pr}{\mathbb P}

\newcommand{\ca}{{\mathcal A}}

\newcommand{\cc}{{\mathcal C}}

\newcommand{\cf}{{\mathcal F}}

\newcommand{\ck}{{\mathcal K}}

\newcommand{\ii}{{\mathbb I}}
\newcommand{\I}{{\mathbb I}}

\begin{document}

\author{Valentas Kurauskas  and  Mindaugas Bloznelis \\ \\ Vilnius University}


\maketitle


\begin{abstract}
%
Given positive integers $n$ and $m$, and a probability measure $P$ on $\{0, 1, \dots, m\}$,
the random intersection graph $G(n,m,P)$ on vertex set $V = \{1,2, \dots, n\}$ and
with attribute set $W  = \{w_1, w_2, \dots, w_m\}$ is defined as follows.
Let $S_1, S_2, \dots, S_n$ be independent random subsets of $W$ such that for 
any $v \in V$ and any $S \subseteq W$ we have
$\pr( S_v = S ) = {P(|S|)} / {\binom m {|S|}}$.
The edge set of $G(n,m,P)$ consists of those pairs $\{u,v\} \subseteq V$ for 
which $S_u \cap S_v \ne \emptyset$.

We study the asymptotic order of the clique number $\omega(G(n,m,P))$ 
 of  sparse
 random intersection graphs.
For instance, for $m = \Theta(n)$ 
we show that 
the maximum clique is of 
size
\begin{displaymath}
(1-\alpha/2)^{-\alpha/2}n^{1-\alpha/2}(\ln n)^{-\alpha/2}(1+o_P(1))
\end{displaymath}
in the case where the vertex degree distribution is a
power-law with exponent $\alpha \in (1;2)$, and it is 
of 
size
$\frac {\ln n} {\ln \ln n}(1+o_P(1))$
in the case where the degree distribution has a finite variance. 
In each case 
there is
a polynomial algorithm which finds a clique of size 
$\omega(G(n,m,P)) (1-o_P(1))$. 

\end{abstract}

\bigskip

\emph{keywords: clique, random intersection graph, greedy algorithm, 
complex network, 
power-law, clustering}

\small {
  \bigskip

  corresponding author:

  Valentas Kurauskas, Faculty of Mathematics and Informatics, Didlaukio 47, LT-08303 Vilnius, Lithuania

  \bigskip

  email:

  Mindaugas Bloznelis: mindaugas.bloznelis@mif.vu.lt
  
  \smallskip

  Valentas Kurauskas: valentas@gmail.com

  }

 \newpage

\section{Introduction}
\label{sec.intro}

Bianconi and Marsili observed in 2006 \cite{bianc06} 
that ``scale-free'' real networks can have 
very large cliques;
they gave an argument suggesting that the rate of divergence is polynomial 
if the degree variance is unbounded~\cite{bianc06}.
In a more precise analysis, Janson, \L uczak and Norros \cite{jln10} 
showed exact asymptotics for the clique number in a power-law random graph model where 
edge probabilities are proportional
to the product of weights of their endpoints.

Another 
feature of a real network
that may affect  formation of cliques
is the clustering property: the probability of a link between two randomly chosen vertices 
increases dramatically after we learn 
about the presence of their common neighbour.
An interesting question is  whether and how  the clustering property
is related to the clique number.

With conditionally independent edges, 
the random graph of \cite{jln10} does 
not have the clustering
property and, therefore, can not 
explain such a relation.

In the present paper we address this question  by
showing precise asymptotics  
for the clique number of a related random intersection graph model that admits a 
tunable clustering 
coefficient and  power-law degree distribution. We find that the effect of  
clustering  on the clique number only 
shows up for the degree sequences having a 
finite variance. We note that the finite variance is a necessary, but not sufficient 
condition for the clustering coefficient 
to attain
a non-trivial value,
see 
\cite{bloznelis11} and (\ref{2013-09-19}) below.

In the language of hypergraphs, we ask what is the largest intersecting 
family in a random hypergraph on the
vertex set $[m]$, 
where $n$ identically distributed and independent hyperedges have random sizes
distributed according to $P$. 
A related problem for uniform hypergraphs was
considered by 
Balogh, Bohman and Mubayi \cite{bbm2009}. 
Although the motivation  and the approach of \cite{bbm2009} are 
different from ours, 
the result of \cite{bbm2009}  yields  the clique number, for a particular
class of random intersectipon graphs based on the subsets having the 
same (deterministic) number of elements.

The random intersection graph model was introduced by Karo\'nski, Scheinerman and 
Singer-Cohen in 1999 \cite{karonski99} and further generalised by 
Godehardt and Jaworski \cite{gj03} and others.
With appropriate parameters, it yields graphs that are
sparse \cite{deijfen, bloznelis08}, have a positive clustering coefficient 
\cite{deijfen, bloznelis11} and assortativity  \cite{bloznelis12}. 
We will consider a sequence $\{G(n)\}=\{G(n), n = 1, 2, \dots\}$ of 
random intersection graphs 
$G(n) = G(n, m, P)$, where $P = P(n)$ and $m = m(n)\to+\infty$ as $n\to+\infty$.
Let $X(n)$ denote a random variable distributed 
according to $P(n)$ and define $Y(n) := \sqrt {\frac n m} X(n)$. If 
not explicitly stated otherwise, the limits below will be taken as $n \to \infty$. In this
paper we use the standard notation $o()$, $O()$, $\Omega()$, $\Theta()$, $o_P()$, $O_P()$,
see, for example,~\cite{jlr}.
For positive sequences $(a_n)$, $(b_n)$ we write $a_n \sim b_n$ if $a_n/b_n \to 1$ 
$a_n\ll b_n$ if $a_n/b_n\to 0$.
For a sequence of events $\{\ca_n\}$, we say that $\ca_n$ occurs \emph{whp}, if $\pr(\ca_n) \to 1$.
 
 We will assume in what follows that 
\begin{equation}\label{eq.Y}
    \E Y(n) = O(1).
\end{equation}
This condition ensures that 
the expected number of edges in $G(n)$ is $O(n)$. Hence 
$G(n)$ is sparse. We remark, that if, in addition, 
$Y(n)$ converges in distribution
to an integrable random variable, say $Z$, and $\E Y(n)\to \E Z$, then
$G(n)$ has asymptotic degree distribution $Poiss(\lambda)$, where $\lambda=Z\E Z$, 
see, e.g., \cite{bloznelis11}.
In particular,
if $Y(n)$ has asymptotic square integrable distribution, then  
$G(n)$ has asymptotic square integrable degree distribution too.
Furthermore, if $Y(n)$ has a power-law asymptotic distribution, then  
$G(n)$ has asymptotic power-law degree distribution with the same exponent.

Our first result, Theorem \ref{thm.main}, shows that in the latter case the clique number
diverges polynomially.
In fact, we do not require $Y(n)$  to have a limiting power-law distribution, but consider
a 
condition that only involves the tail of $Y(n)$. Namely, we assume that for some 
$\alpha>0$  
and some slowly varying function $L$ 
there is  $0<\eps_0<0.5$ such that
for each sequence $x_n$ with $n^{1/2 - \eps_0} \le x_n \le n^{1/2+\eps_0}$ we have
\begin{equation}\label{eq.pldef2}
  \pr\left (Y(n) \ge x_n \right) \sim L(x_n) x_n^{-\alpha}.
\end{equation}
We recall that a function $L: \mathbb{R_+} \to \mathbb{R_+}$ is called \emph{slowly varying} 
if $\lim_{x \to \infty} L(t x) / L(x) = 1$ for any $t > 0$.

\begin{theorem} \label{thm.main} Let $1<\alpha<2$.
   Assume that  $\{G(n)\}$ is a sequence of random 
 intersection graphs satisfying (\ref{eq.Y}), (\ref{eq.pldef2}).
   Suppose that for some $\beta > \max \{2-\alpha, \alpha - 1\}$ 
we have $m = m(n) = \Omega (n^\beta)$.
   Then the clique number of $G(n)$ is
\begin{equation}\label{2013-09-17}
   \omega(G(n)) =  (1+o_P(1))\left(1 -  \alpha/2 \right) ^ {- \alpha/2}K(n)
\end{equation}
 where
    \[
    K(n) =  L\left( (n \ln n)^{1/2}\right) n ^ {1-\alpha / 2} (\ln n)^{-\alpha/2}.
    \]
\end{theorem}
%
We remark that adjacency relations of neighbouring vertices of a random intersection graph
are statistically dependent events and this dependence is not negligible for $m=O(n)$.
Theorem \ref{thm.main} says that in the case where the asymptotic degree distribution has 
infinite second moment ($\alpha<2$), the asymptotic order (\ref{2013-09-17}) 
of a power-law random intersection graph is the same as that
of 
the  related model of  
\cite{jln10} which has conditionally independent edges.
Let us mention that the
lower bound for the clique number
$\omega(G(n))$  is obtained using a simple and elegant argument of \cite{jln10}, 
 which is not sensitive to the statistical dependence of edges of $G(n)$.
 To show the matching upper 
 bound we developed
another  approach 
based on  a 
result of 
Alon, Jiang, Miller and Pritkin \cite{alon03}
in Ramsey theory.

In the case where the (asymptotic) degree distribution has a finite second moment
we not only  find the asymptotic 
order of $\omega(G(n))$, but also describe the structure of 
a maximal clique. To this aim, 
it is convenient to interpret attributes 
$w\in W$ as colours.
The set of vertices $T(w)=\{v\in V:\, w\in S_v\}$ induces a clique in $G(n)$ which
we denote (with some ambiguity of notation) $T(w)$. We say that every edge
of $T(w)$ receives colour $w$ and call this clique {\it monochromatic}. Note that
$G(n)$ is covered by the union of monochromatic cliques  $T(w)$, $w\in W$.  
We denote the size of the largest monochromatic clique by $\omega'(G(n))$. Clearly,
$\omega(G(n))\ge \omega'(G(n))$.

Denote $x \vee y = \max(x,y)$.
The next theorem shows that the largest clique
is a monochromatic clique (plus possibly a few extra vertices). 
%
%
%
%
 \begin{theorem}\label{thm.finite}
  Assume that $\{G(n)\}$ is a sequence of 
  random intersection graphs satisfying (\ref{eq.Y}). Suppose that  $Var(Y(n)) = O(1)$. 
  Then 
   \[
     \omega(G(n)) = \omega'(G(n))+ O_P(1).
    \]
    If, in addition, for some positive sequence $\{\eps_n\}$ converging to zero
we have
    \begin{equation} \label{eq.uni1}
        n \pr(Y(n) > \eps_n n^{1/2}) \to 0
    \end{equation}
    then, for an absolute constant $C$,
    \[
    \pr\left (\omega(G(n)) \le C \vee (\omega'(G(n)) + 3) \right) \to 1.
    \]
\end{theorem}
%
%
%
The 
condition (\ref{eq.uni1}) is not very restrictive. 
It is satisfied by  uniformly 
square integrable sequences $\{Y(n)\}$. In particular, (\ref{eq.uni1})
holds if $\{Y(n)\}$ converges in distribution to a square integrable 
random variable, say $Y_*$, and $\E Y^2(n)$ converges to $\E Y_*^2$.

Next, we evaluate the size of the largest monochromatic clique. For this purpose
we relate the random intersection graph to the balls into bins model. Let
every vertex $v\in V$ throw $X_v:=|S_v|$ balls into the bins $w_1,\dots, w_m$ 
uniformly at random,
subject to the condition that every bin receives at most one ball from each vertex.
Then $\omega'(G(n))$ counts the maximum number of balls contained in a bin.
 Let $M(N, m)$ denote the maximum number of balls  contained in any of $m$ bins after
$N$ balls were thrown into $m$ bins uniformly and independently at random.
Our next result says that the probability distribution of $\omega'(G(n))$ can be approximated
by that of $M(N, m)$, with $N\approx n\E X(n)=\E(X_1+\dots+X_n)$.
The asymptotics of $M(N, m)$ are well known, see, e.g., 
Section~6 of Kolchin et al \cite{kolchin}.

Denote by $d_{TV}(\xi,\eta)= 2^{-1}\sum_{i\ge 0}|\pr(\xi =i) - \pr(\eta=i)|$ 
the total variation distance between probability distributions
of non-negative integer valued random variables $\xi$ and $\eta$. 
%
%
%
%
%
%
%
\begin{theorem}\label{thm.omega.mono}
Assume that  $\{G(n)\}$ is a sequence of 
random intersection graphs satisfying
$\E Y = \Theta(1)$ and $Var(Y) = O(1)$.  Then
\[
d_{TV}(\omega'(G(n)), M(\lfloor (mn)^{1/2}\E Y (n) \rfloor, m)) \to 0.
\]
\end{theorem}


\begin{remark} \label{rmk.omega.mono} 
For $n, m\to+\infty$ the relations  
$\E Y = \Theta(1), Var Y = O(1)$ imply $n=O(m)$. In particular, the conditions of 
Theorem~\ref{thm.omega.mono} rule out the case $m=o(n)$.
\end{remark}
%
%
%
Let us summarize our results about the clique number of a sparse random intersection graph
$G(n)$ with a square integrable (asymptotic) degree distribution. 
We note that the conditional probability
(called the clustering coefficient of $G(n)$)
\begin{equation}\label{2013-09-19}
 \pr(v_1\sim v_2|v_1\sim v_3, v_2\sim v_3)\approx (n/m)^{1/2}\E Y(n)/\E Y^2(n)
\end{equation}
only attains a non-trivial value for $m=\Theta(n)$ and $\E Y^2(n)=\Theta(1)$. (Here $u \sim v$
is the event that $u$ and $v$ are adjacent in $G(n)$, i.e., $uv \in E(G(n))$.)
In the latter case 
 Theorems~\ref{thm.finite} and 
\ref{thm.omega.mono}
 together with the asymptotics for $M(N,m)$ (Theorem~II.6.1 of \cite{kolchin}), 
imply 
that
\[
  \omega(G(n)) = \frac {\ln n} {\ln \ln n} \left( 1 + o_P(1) \right).
\]
 In contrast, the clique number of a sparse  Erd\H os-R\'enyi random graph $G(n,c/n)$ 
is at most $3$, and in the model of \cite{jln10}, with square 
integrable asymptotic degree distribution,
 the largest clique has at most 4 vertices.

Each of our main results, Theorem~\ref{thm.main} and Theorem~\ref{thm.finite},
have corresponding 
simple polynomial algorithms that construct a clique of the optimal order whp.
For  a power-law graph with $\alpha \in (1;2)$, it is the greedy algorithm 
of \cite{jln10}: sort  vertices 
in descending order according to their degree; traverse vertices in 
that order and `grow' a clique, 
by adding a vertex
if it is
connected to each vertex in the current clique.
For  
a graph with a finite degree variance 
we propose even simpler algorithm:
 for each pair of adjacent vertices, 
take any maximal clique
formed by that pair and their common neighbours. 
Output the biggest maximal clique found in this way. 
More details and analysis of each of the algorithms are given 
in Section~\ref{sec.algo} below.

In practical situations a graph may be assumed to be distributed as a random 
intersection graph,
but information about the subset size distribution may not be available. 
In such a case, 
instead of condition (\ref{eq.pldef2}) 
for the 
tail of the normalised subset size $Y(n)$, we may consider a similar condition 
for the tail of the degree $D_1(n)$
of the vertex $1\in V$ in $G(n)$: 
there are constants $\alpha' > 1, \eps' > 0$ and
a slowly varying function $L'(x)$ such that     for any 
sequence $t_n$ with $n^{1/2 - \eps'} \le t_n \le n^{1/2+\eps'}$
    \begin{equation} \label{eq.degheavytail}
        \pr(D_1(n) \ge t_n) \sim L'(t_n) t_n^{-\alpha'}.
    \end{equation}
The following lemma shows that, subject to an additional assumption, there is
equivalence between conditions (\ref{eq.pldef2})  and (\ref{eq.degheavytail}) .
%
%
\begin{lemma} \label{lem.degrees} 
    Assume that  $\{G(n)\}$ is 
a sequence of random intersection graphs such that 
    for some $\eps > 0$ we have
    \begin{equation} \label{eq.uniform}
        \E Y(n) \I_{Y(n) \ge n^{1/2 - \eps}} \to 0.
    \end{equation}
    Suppose that either $(\E Y(n))^2$ or $\E D_1(n)$ converges to a positive number, say, $d$. 
    
    Then both limits exist and are equal,  $\lim\E D_1(n)=\lim(\E Y(n))^2=d$.
Furthermore, the condition (\ref{eq.degheavytail}) holds if and only if 
(\ref{eq.pldef2}) holds. In that case,
    $\alpha' = \alpha$ and $L'(t) = d^{\alpha/2} L(t)$.
\end{lemma}
%
%
%
Thus, under a mild additional assumption (\ref{eq.uniform}), condition (\ref{eq.pldef2}) 
of Theorem~\ref{thm.main} can be replaced by  (\ref{eq.degheavytail}).
Similarly, the condition  $Var Y(n)=O(1)$ of  Theorem~\ref{thm.finite}
can be replaced by 
the condition $Var D_1(n)=O(1)$. 
\begin{lemma} \label{lem.degvar}
    Assume that  $\{G(n)\}$ is a sequence of random intersection graphs and for some
positive sequence $\{\eps_n\}$ converging to zero we have
    \begin{equation} \label{eq.uni2}
      \E Y^2(n) \I_{Y(n) > \eps_n n^{1/2}} \to 0. 
  \end{equation}
    Suppose that either $\E Y(n) = \Theta(1)$ or $\E D_1(n) = \Theta(1)$. Then
    \begin{align}
        \E D_1(n) &= (\E Y(n))^2  + o(1) \label{eq.ED12}
        \\ Var D_1(n) &= (\E Y(n))^2 (Var Y(n) + 1)   + o(1). 
\label{eq.VarD}
    \end{align}
\end{lemma}

\smallskip


Cliques of random intersection graphs 
have been studied in  
\cite{karonski99}, where edge density thresholds 
for emergence of  small (constant-sized) 
cliques were determined, and in \cite{rybstark2010}, where the Poisson approximation
to the distribution 
of the number of small cliques was established. The clique number was studied in 
\cite{nrs2012}, see also \cite{btu09},
in the case, where $m\approx n^{\beta}$, for some $0<\beta<1$. We note that in
the papers \cite{karonski99}, \cite{rybstark2010}, \cite{nrs2012} a particular random 
intersection graph with the binomial distribution $P\sim Bin(p, m)$ was considered.


The rest of the paper is organized as follows. In Section~\ref{sec.power_law} we 
study sparse random power-law intersection graphs with index 
$\alpha \in (1;2)$, 
introduce the result on ``rainbow'' cliques in extremal combinatorics 
(Lemma~\ref{lem.rainbow}) and prove Theorem~\ref{thm.main}.
In Section~\ref{sec.finite.variance} we relate our model 
to the balls and bins model and prove Theorem~\ref{thm.finite}. 
In Section~\ref{sec.algo} we present and analyse algorithms 
for finding large cliques in $G(n,m,P)$. 
In Section~\ref{sec.tailequivalence} we prove 
Lemmas~\ref{lem.degrees}~and~\ref{lem.degvar}.
Finally we give some concluding remarks.

\section{Power-law intersection graphs}
\label{sec.power_law}

\subsection{Proof of Theorem~\ref{thm.main}}
\label{subsec.prelim}

We start with introducing some notation. Given  a family of 
subsets $\{S_v, v\in V'\}$ of an attribute set $W'$, we denote $G(V',W')$
 the 
 \emph{intersection graph} on the vertex set $V'$ defined by this family: $u,v\in V'$ are adjacent
(denoted $u\sim v$)  
whenever $S_u\cap S_v\not=\emptyset$. 
We say that an attribute $w\in W'$ \emph{covers} the edge $u\sim v$ of $G(V',W')$ 
whenever $w\in S_u\cap S_v$. In this case we also say that the edge $u\sim v$ receives 
colour $w$.
In particular, an attribute $w$ covers all edges
of the (monochromatic) clique  subgraph $T_w$ of $G(V',W')$ 
induced by the vertex set $T_w=\{v\in V':\, w\in S_v\}$. Given a graph $H$, we say that 
$G(V',W')$ \emph{contains a rainbow} $H$ if there is a subgraph $H'\subseteq G(V',W')$ 
isomorphic to $H$ such that
every edge  of $H'$ can be prescribed 
an attribute that covers this edge so that all prescribed attributes are different.

We denote by $e(G)$ the size of the set $E(G)$ of edges of a graph $G$.
Given two graphs $G=(V(G),E(G))$ and $R=(V(R),E(R))$  
we denote by $G\vee R$ the graph on 
vertices 
$V(G)\cup V(R)$ and
with edges 
$E(G)\cup E(R)$. In what follows we 
assume that $V(G)=V(R)$ if not mentioned otherwise. 
Let $t$ be a positive integer and let $R$ be a non-random graph on the vertex set $V'$. 
Assuming that subsets $S_v$, $v\in V'$ are drawn at random, introduce the event 
$Rainbow(G(V',W'),R,t)$ that the graph $G(V',W')\vee R$  has a clique $H$ of 
size
$|V(H)|=t$ with 
the property that every edge  of the set $E(H) \setminus E(R)$ can be prescribed 
an attribute that covers this edge so that all prescribed attributes are different.

In the case where every vertex $v$ of 
the random intersection graph $G(n,m,P)$ includes attributes independently at random
with probability $p=p(n)$, the size $X_v:=|S_v|$ of the attribute set has binomial distribution 
$P\sim Binom(m,p)$. We denote such graph
$G(n,m,p)$ and call it a \emph{binomial} random intersection graph.
We note that for 
$mp\to+\infty$ the sizes $X_v$ of random sets are concentrated 
around their mean value 
$\E X_v=mp$. An application of   Chernoff's bound (see, e.g., \cite{cmcd98})
\begin{equation} \label{eq.chernoff}
\pr( \left| B - mp \right| > \eps mp) \le 2 e^{-\frac13 {\eps^2 mp}},
\end{equation}
where $B$ is a binomial random variable $B\sim Binom(m,p)$ and $0<\eps < 3/2$,
implies 
\begin{equation} \label{eq.barB}
\pr(\exists v \in [n] : \left| X_v - mp \right| > y) \le n \pr( | X_v - mp| > y) \to 0
\end{equation}
for any $y = y(n)$ such that $y / \sqrt {mp \ln n} \to \infty$ and $y/(mp)<3/2$.

We write $a\wedge b =\min\{a,b\}$ and $a\vee b=\max\{a,b\}$.

Let us prove Theorem~\ref{thm.main}.  For every member $G(n)=G(V,W)$ of a sequence 
$\{G(n)\}$ 
satisfying conditions of 
Theorem~\ref{thm.main} and a number
 $\eps_1\in (0,\eps_0)$ 
define the subgraphs $G_i\subseteq G(n)$, $i=0,1,2$, induced by the vertex sets
\begin{align*}
 & V_0 = V_0(n) = \{v\in V(G(n)):\, X_v < \theta_1\}; \\
 & V_1 = V_1(n) =  \{v\in V(G(n)):\, \theta_1 \leq X_v \leq \theta_2\};\\
 & V_2 = V_2(n) = \{v\in V(G(n)):\, \theta_2 < X_v\},
\end{align*}
respectively. 
Here $X_v=|S_v|$ denotes the size of the attribute set prescribed to a vertex $v$ and
the numbers
\begin{align*}
&\theta_1 = \theta_1(n) = m^{1/2} n^{-\eps_1}; &\theta_2 = \theta_2(n) = \left( (1-\alpha/2) m \ln n + m e_1 \right)^{1/2},
\end{align*}
with
$e_1 = e_1(n) = \max(0, \ln L( (n \ln n)^{1/2}))$. Note that $e_1\equiv 0$ for $L(x)\equiv 1$.
We have
$V=V_0\cup V_1\cup V_2$ and $V_i\cap V_j=\emptyset$ for $i\not= j$.
Theorem~\ref{thm.main} follows from the three lemmas below.
Let $K = K(n)$ be as in Theorem~\ref{thm.main}.
The first lemma gives a lower bound for the clique number of $G(n)$. 
\begin{lemma}\label{lem.G2} For any $m = m(n)$ 
 \[
    \omega(G_2) = |V_2| (1 - o_P(1)) =(1-o_P(1))\left(1 -  \alpha/2 \right) ^ {- \alpha/2}K.
  \]
\end{lemma}
The next two lemmas provide an upper bound.
\begin{lemma} \label{lem.G0}
    Suppose there is $\beta > \alpha - 1$ such that
    $m = \Omega(n^\beta)$. If $\eps_1 < \frac \beta 6$ then there is $\delta > 0$ such that
    \[\pr \left( \omega(G_0) \ge n^{1-\alpha/2 - \delta} \right) \to 0.\]
\end{lemma}

\begin{lemma} \label{lem.G1} 
    Suppose there is  $\beta > 2 - \alpha$ such that  $m = \Omega(n^\beta)$. If $\eps_1 < \frac {\beta - 2 + \alpha} {24}$ then
  \[\omega(G_1) = o_P(K).\]
\end{lemma}

\begin{proofof}{Theorem~\ref{thm.main}}
  We choose $0<\eps_1 < \min\{(\alpha - 1)/6, (\beta - 2 + \alpha)/24, \eps_0\}$.
The theorem follows from the inequalities
$\omega(G_2)\le \omega(G) \le \omega(G_0) + \omega(G_1) + \omega(G_2)$ 
and  Lemmas~\ref{lem.G2}, \ref{lem.G0}~and~\ref{lem.G1}.
\end{proofof}
\vskip 0.1 cm


\subsection{Proof of Lemma \ref{lem.G2}}\label{subsec.large}

In this section we use ideas from \cite{jln10} to give a lower bound on the clique number.
We first note the following 
auxiliary
facts.
\begin{lemma}\label{lem.lambertW}
  Suppose $a=a_n, b = b_n$ are sequences of positive reals such that $0 < \ln 2 b + 2 a \to +\infty$. 
  Let $z_n$ be the positive root of 
  \begin{equation} \label{eq.lambertW}
    a - \ln z - b z^2 = 0.
  \end{equation}
  Then $z_n \sim \sqrt{\frac{2 a + \ln (2 b)} {2b}}$.
\end{lemma}
\begin{proof}
  Changing the variables $t = 2 b z^2$ we get
  \[
   t + \ln(t) = 2 a + \ln (2b).
  \]
  From the assumption it follows that $t + \ln t \sim t$ and therefore $z_n \sim \sqrt{\frac{2a + \ln(2b)}{2b}}.$
\end{proof}

\begin{lemma}[\cite{galambosseneta}] \label{lem.sv}
Let $x\to+\infty$. 
For any slowly varying function $L$  and any $0<t_1<t_2<+\infty$ the convergence
$L(tx)/L(x)\to 1$ is uniform in $t\in [t_1,t_2]$. Furthermore,  
we have
 $\ln  L(x) = o(\ln x)$.
\end{lemma}

\begin{proofof}{Lemma~\ref{lem.G2}}
Write $N = |V_2|$ and let
\[
v^{(1)}, v^{(2)}, \dots, v^{(N)}
\]
be the vertices of $V_2$ listed in an arbitrary order. 

Consider a greedy algorithm for finding a clique in $G$ proposed by Janson, \L uczak  and Norros \cite{jln10} (they use descending ordering by the set sizes, see also Section~\ref{sec.algo}). Let $A^{0} = \emptyset$. In the step $i = 1,2,\dots, N$ 
let $A^{i} = A^{i-1} \cup \{v^{(i)}\}$ if $v^{(i)}$ is incident to each of the vertices $v^{(j)}$, $j = 1,\dots,i-1$. Otherwise,
let $A^{i} = A^{i-1}$. 
This algorithm produces a clique $H$ on the set of vertices $A^N$, and $H$ demonstrates that $\omega(G_2) \ge |A^N|$.  

Write $\theta = \theta_2$ and
let $L_{\theta} = V_2 \setminus A^{N}$ be the set of 
vertices that failed to be added to $A^N$. 
We will show that 
\[
\frac {|L_{\theta}|} {N \vee 1} = o_P(1)
\]
and
\[
N = \left(1 - \alpha/2 \right)^{-\alpha/2} L\left ( (n \ln n)^{1/2} \right)  (\ln n)^{-\alpha/2} n^{1-\alpha/2} ( 1 - o_P(1)).
\] 
From (\ref{eq.pldef2}) we obtain for $N \sim Binom(n, q)$ with 
$q = \pr (X_n >  \theta)$
  \begin{align*}
      \E N = nq &= n\pr \left( (m/n)^{1/2} Y_n > \theta \right) 
  \\ &\sim L \left( (n/m)^{1/2} \theta \right) n^{1-\alpha/2} m^{\alpha/2} \theta^{-\alpha}
  \\ &\sim \left( 1 - \alpha/2 \right)^{-\alpha/2} L(\sqrt{n \ln n}) (\ln n)^{-\alpha/2} n^{1-\alpha/2} .
\end{align*}
Here we used  $L((n/m)^{1/2} \theta) \sim L(\sqrt {n \ln n})$
and $\ln L (\sqrt {n \ln n}) = o(\ln n)$, see Lemma~\ref{lem.sv}. 
Furthermore, by the concentration property of the binomial distribution, see, e.g.,
 (\ref{eq.chernoff}), 
we have $N=(1-o_P(1))\E N$.

The remaining bound $|L_{\theta}|/(N \vee 1) \le |L_{\theta}|/(N+1) = o_P(1)$ 
follows
 from 
 the bound
$\E (L_{\theta}/(N+1))=o(1)$, which is shown below.

  Let $p_1$ be the probability that two random 
independent subsets of $W=[m]$ of 
size $\lceil \theta \rceil$
do not intersect. 
  The number of vertices in $L_{\theta}$ is at most the number of pairs in $x,y \in V_2$ where $S_x$ and $S_y$ do not intersect. 
  Therefore by the first moment method
\[
  \E \frac{|L_{\theta}|}{N+1} = \E \E\left(\frac{|L_{\theta}|}{N+1} \Bigl| N\right)
\le 
\E \E \left ( \frac {\binom N 2 p_1} {N+1}  \Big| N \right) \le \frac {p_1\E N } 2,
  \]
where
  \[
  p_1 = \frac {\binom {m-\theta} \theta} {\binom m \theta} 
\le 
\left( 1-\frac \theta m \right)^\theta 
\le 
e^{-\theta^2/m}.
  \]
Now it is straightforward to check that for some constant $c$ we have 
$p_1 \E N  \le c (\ln n)^{-\alpha/2}  \to 0$. 
This completes the proof. 

Let us briefly explain the intuition for the choice of $\theta$. 
For simplicity assume $L(x) \equiv 1$ so that $e_1 = 0$. 
Could the same method yield a bigger clique if $\theta_2$ is smaller? 
We remark that the product $p_1\E N$ as well as its upper bound 
$n^{1-\alpha/2}m^{\alpha/2}\theta^{-\alpha}e^{-\theta^2/m}$ (which we used above)
are decreasing functions of $\theta$.
Hence, if we wanted this upper bound to be $o(1)$ then $\theta$ should be at least as large as the 
solution to the equation
\[
 n^{1-\alpha/2} m^{\alpha/2} \theta^{-\alpha} e^{-\theta^2/m} = 1
\]
or, equivalently, to the equation
\begin{equation}\label{eq.theta}
 \alpha^{-1} \ln n + \frac 1 2 \ln (m/n) - \ln \theta - \frac {\theta^2} {\alpha m} = 0.
\end{equation}
After we write the latter relation in the form
(\ref{eq.lambertW}) where
$a = \alpha^{-1} \ln n + (1/2) \ln (m/n)$ and $b =  (\alpha m)^{-1}$ 
satisfy
$b e^{2a} = \alpha^{-1} n^{\frac 2 \alpha - 1} \to +\infty$, we obtain from 
Lemma~\ref{lem.lambertW} that
the solution $\theta$
of (\ref{eq.theta}) satisfies
\[
\theta \sim \sqrt {\frac {(2/\alpha) \ln n - \ln (n/m) + \ln (2/\alpha m)} {2/\alpha m}} \sim \sqrt{(1-\alpha/2) m \ln n}.
\]
\end{proofof}


\subsection {Proof of Lemma \ref{lem.G0}}\label{subsec.small}

Before proving Lemma  \ref{lem.G0} we collect some preliminary 
results.
\begin{lemma}\label{lem.karonski}
  Let $h$ be a positive integer.
Let $\{G(n)\}$ be a sequence of binomial random intersection graphs
$G(n) = G(n,m,p)$, were $m = m(n)$ and $p = p(n)$ satisfy
  $p n^{1/(h-1)}m^{1/2} \to a \in \{0,1\}$. Then
  \[
  \pr(G \text{ contains a rainbow $\ck_h$}) \to a.
  \]
\end{lemma}
\begin{proof} 
  The case $a=1$ follows from Claim 2 of \cite{karonski99}. For the case $a=0$ we have,
 by the first moment method,
  \begin{align*}
  \pr(G \text{ contains a rainbow $\ck_h$}) &\le \binom n h (m)_{\binom h 2} p^{2\binom h 2} 
  \\ &  \le 
        \left( n^{1/(h-1)} m^{1/2} p \right)^{h(h-1)} \to 0.
  \end{align*}
\end{proof}
\medskip

Next is an upper bound
for the size $\omega'(G)$ of the largest monochromatic clique.
\begin{lemma}\label{lem.ballsbins} Let $1<\alpha<2$.
Assume that  $\{G(n)\}$ is a sequence of random 
 intersection graphs satisfying (\ref{eq.Y}), (\ref{eq.pldef2}).
   Suppose that for some $\beta > \alpha - 1$ 
we have $m=\Omega (n^\beta)$.
  Then there is a constant $\delta > 0$ such that 
$\omega'(G(n)) \le n^{1-\alpha/2 - \delta}$
    whp.
\end{lemma}

\begin{proof}
Let $X = X(n)$  and $Y = Y(n)$ be defined as in (\ref{eq.Y}).
  Since for any  $w \in W$ and  $v \in V$ 
  \[
  \pr (w \in S_v) 
= 
\sum_{k=0}^\infty \frac k m \pr (|S_v| = k) = \frac {\E X} m = \frac{ \E Y} {\sqrt{mn}},
  \]
and the  number of elements of  the set $T_v=\{v:\, w\in S_v\}$ is binomially distributed 
  \begin{equation}\label{eq.mono}
    |T_w| \sim Binom \left(n, \frac {\E Y} {\sqrt {m n}}\right),
  \end{equation}
  we have, for any positive integer $k$
 \[
 \pr (|T_w| \ge k) 
\le 
\binom n k \left ( \frac {\E Y} {\sqrt{mn}}\right)^k 
\le 
\left(\frac{en}{k}\right)^k\left( \frac {\E Y} {\sqrt{mn}} \right)^k 
\le  
\left( \frac {c_1} k \sqrt {\frac n m}\right)^k
 \]
for $c_1 = e \sup_n \E Y$. Therefore, by the union bound,
\[
\pr \left( \omega'(G(n)) \ge k \right) \le m \left(\frac {c_1} k \sqrt{\frac n m}\right)^k.
\]
Fix $\delta$ with $0 < \delta < \min ( (\beta-\alpha + 1)/4, 1- \alpha/2, \beta/2)$.
We have 
\begin{align*}
    &\pr \left( \omega'(G(n)) \ge n^{1-\alpha/2 - \delta} \right) \le m \left( c_1 n^{\alpha/2 - 1/2 + \delta} m^{-1/2} \right)^{\left \lceil n^{1-\alpha/2 - \delta} \right \rceil}
    \\  &= m^{1 - (\delta/\beta) \left \lceil n^{1-\alpha/2 - \delta} \right \rceil} \left( c_1 n^{\alpha/2 - 1/2 + \delta} m^{-1/2 + \delta/\beta} \right)^{\left \lceil n^{1-\alpha/2 - \delta} \right \rceil}\to 0
\end{align*}
since $m \to \infty$, $n^{1-\alpha/2 - \delta} \to \infty$ 
and $m=\Omega(n^{\beta})$ implies
\[
n^{\alpha/2 - 1/2 + \delta} m^{-1/2 + \delta/\beta} 
\to 0.
\]
\end{proof}

\vskip 0.1 cm

The last and the most important fact we need relates the maximum clique 
size with the maximum rainbow clique 
size in an  intersection graph. 
An edge-colouring of a graph is called
$t$-good if each colour appears at most $t$ times at each vertex.
We say that an edge-coloured graph contains a rainbow copy of $H$ if it
 contains a subgraph isomorphic to $H$ with all edges receiving different colours.
\begin{lemma}[\cite{alon03}] \label{lem.rainbow}
  There is a constant $c$ such that
  every $t$-good coloured complete graph on more than $\frac{c t h^3} {\ln h}$ vertices contains a rainbow copy of $\ck_h$.
\end{lemma}

\bigskip

\begin{proofof}{Lemma~\ref{lem.G0}}
Fix an integer $h > 1 + \frac 1 {\eps_1}$ and denote 
$t=n^{1- \alpha/2 - \delta}$
and $k=\lceil \frac{c t h^3} {\ln h}\rceil$, where positive constants $\delta$ and $c$ are from Lemmas
\ref{lem.ballsbins}
and \ref{lem.rainbow}, respectively.

We first show that  
\begin{equation}\label{eq.npkh}
     \pr(G_0 \text{ contains a rainbow } \ck_h) =o(1).
   \end{equation}
We note that for the  binomial intersection graph ${\tilde G}=G(n,m,p)$  with 
$p = p(n)= m^{-1/2} n^{-\eps_1} + m^{-2/3}$ Lemma~\ref{lem.karonski} implies
\begin{equation}\label{eq.npkh+}
     \pr({\tilde G} \text{ contains a rainbow } \ck_h) = o(1).
   \end{equation}
 Let ${\tilde S}_v$ (respectively $S_v$), $v\in V$,
denote the random 
subsets prescribed to vertices of ${\tilde G}$ (respectively $G(n)$).
Given the set sizes $|S_v|, |{\tilde S}_v|$, $v\in V$, satisfying $|{\tilde S}_v|>\theta$, 
for each $v$, 
we 
couple  the random sets of $G_0$ and ${\tilde G}$ so that $S_v\subseteq {\tilde S}_v$, 
for all $v\in V_0$. Now $G_0$ becomes a subgraph of ${\tilde G}$ and 
(\ref{eq.npkh}) follows from (\ref{eq.npkh+}) and the fact that 
$\min_{v}|{\tilde S}_v|>\theta$ whp, 
see (\ref{eq.barB}).

Next, we colour every edge $x\sim y$  of $G_0$ by an arbitrary element of
  $S_x \cap S_y$ and observe that the inequality $\omega'(G(n))\le t$ (which holds
with probability $1-o(1)$, by Lemma~\ref{lem.ballsbins}) implies that
the colouring obtained is $t$-good. Furthermore, by Lemma~\ref{lem.rainbow}, every 
$k$-clique of 
$G_0$ contains a rainbow clique; however the probability of the 
latter event is negligibly small  by (\ref{eq.npkh}). We conclude that 
$\pr(\omega(G_0)\ge k)=o(1)$ thus proving the lemma.
\end{proofof}



\subsection {Proof of Lemma~\ref{lem.G1}}
\label{subsec.intermediate}

We start with a combinatorial lemma which is of independent interest.
\begin{lemma}\label{lem.DR}
Given positive integers $a_1,\dots, a_k$, let 
$\{A_1, \dots, A_k\}$ be a family of subsets of $[m]$ of sizes $|A_i|=a_i$.
Let $d\ge k$ and let $S$ be a random subset of $[m]$ of size $d$.
Suppose that $a_1+\dots+a_k\le m$. Then the probability
  \begin{equation}\label{eq.DR}
    \pr \left( \{S \cap A_1, \dots, S\cap A_k\} \text{ has a system of distinct representatives} \right)
  \end{equation}
  is maximised when $\{A_i\}$ are mutually disjoint.
\end{lemma}

\begin{proof}
  Call any of $\binom m d$ possible outcomes $c$ for $S$ a configuration.
  Given $\cf = \{A_1, \dots, A_k\}$ let $\cc_{DR}(\cf)$ be the set of  configurations 
$c$ 
such that $c \cap \cf = \{c \cap A_1, \dots, c\cap A_k\}$ has 
a system of distinct representatives. Write
\[
p(\cf) = \sum_{1 \le i < j \le k} |A_i \cap A_j|.
\]
Suppose the claim is false. Out of all families
that maximize (\ref{eq.DR}) pick a family $\cf$ with smallest $p(\cf)$.
Then $p(\cf) > 0$ and we can assume that
there is an element $x \in [m]$ 
such that  $x \in A_1 \cap A_2$. 
  Since $\sum_{i=1}^k |A_i| \le m$, there is an element $y$ in 
the complement of $\bigcup_{A\in \cf} A$.

  Define $A_1' = (A_1 \setminus \{x\}) \cup \{y\}$ and 
consider the family  $\cf' = \{A_1', A_2, \dots, A_k\}$.
Observe that the family of configurations $\cc = \cc_{DR}(\cf) \setminus \cc_{DR}(\cf')$ 
has the following property: for each  $c \in \cc$ we have
$x\in c$  
and
it is not possible to find a set of distinct representatives for $c \cap \cf$ 
where $A_1$ is matched
with an element other than $x$ (indeed such a set of distinct representatives, 
if existed,
would imply $c\in \cc_{DR}(\cf')$). Consequently, there is a set of distinct representatives
for sets $c \cap A_2,\dots, c\cap A_k$ which does not use $x$. Since the latter set of 
distinct representatives together with $y$ is a set of distinct representatives
for $c\cap \cf'$, we conclude that $c\not\in \cc_{DR}(\cf')$ implies $y\notin c$.

Now, for $c \in \cc$, let $c_{xy}=(c\cup\{y\})\setminus\{x\}$ 
be the configuration with  $x$ and $y$ swapped. 
Then $c_{xy} \not \in \cc_{DR}(\cf)$ and  $c_{xy} \in \cc_{DR}(\cf')$,
 because $y\in c_{xy}$ and can be matched with $A_1$. Thus
 each  configuration $c \in \cc$ is 
 assigned
a unique  
configuration $c_{xy} \in \cc_{DR}(\cf') \setminus \cc_{DR}(\cf)$.
This shows that $|\cc_{DR}(\cf')| \ge |\cc_{DR}(\cf)|$. But
$p(\cf') \le p(\cf) - 1$, which contradicts our assumption
about the minimality of~$p(\cf)$.
\end{proof}

\medskip

The next lemma is a version of a result of Erd\H{o}s and R\'{e}nyi about the maximum 
clique of the binomial random graph $G(n,p)$ (see, e.g., \cite{jlr}). 
\begin{lemma} \label{lem.planted}
Let $n\to+\infty$. Assume that probabilities $p_n \to 1$.
Let  $\{r_n\}$ be a positive sequence,
satisfying 
$r_n = o ({\tilde K}^2)$, where ${\tilde K}=\frac { 2 \ln n} {1-p_n}$.

There are positive sequences $\{\delta_n\}$ and $\{\eps_n\}$ 
converging to zero, such that  $\delta_n{\tilde K}\to+\infty$ and 
for 
any
 sequence of 
non-random
graphs $\{R_n\}$ 
with $V(R_n) = [n]$ and $e(R_n) \le r_n$ 
the number $X_n$ of
cliques of size $\lfloor {\tilde K} (1 + \delta_n) \rfloor$ in $G(n,p_n) \vee R_n$ satisfies
\[
 \E X_n \le \eps_n.
\]
\end{lemma}

\begin{proof}
  Write $p=p_n, r=r_n$ and $h = 1-p$. Pick a positive sequence 
$\delta=\delta_n$ so that $\delta_n \rightarrow 0$ and 
$\ln^{-1} n + h +  \frac r {{\tilde K}^2} = o(\delta)$.
  Let $a = \left \lfloor {\tilde K}(1+\delta) \right \rfloor$.
  We have
 \begin{equation}\label{2013-07-27}
  \E X_n 
\le 
\binom n a p^{\binom a 2 - r} 
\le 
\left( \frac {e n} a \right)^a p^{\frac {a(a-1)} 2 - r}=e^{aB},
 \end{equation}
where, by the inequality $\ln p\le -h$, for $n$ large enough,
\begin{align*}
 B
&\le \ln (en/a)  -\left(\frac {a-1} 2 - \frac r a  \right)h
\\& \le \ln n - \frac {ah}2 + \frac {rh}{a}
\le 
 (-1+o(1))\delta\ln n  \rightarrow -\infty.
\end{align*}
\end{proof}

\vskip 0.2 cm


\begin{lemma}\label{lem.plantedrig}
  Let $\{G(n)\}$ be a sequence of binomial random intersection graphs, 
where $m=m_n\to+\infty$
and $p=p_n\to 0$ as $n\to+\infty$.
Let $\{r_n\}$ be a sequence of positive integers.
 Denote  ${\bar K} = 2 e^{mp^2}\ln n$. Assume that  $r_n\ll{\bar K}^2$ and
\begin{equation}\label{2013_08_17.mpK}
mp^2\to+\infty,
\qquad
\ln n\ll mp,
\qquad 
{\bar K}p\to 0,
\qquad
{\bar K}\le n/2.
\end{equation}

  There are positive sequences $\{\eps_n\}, \{\delta_n\}$ 
converging to zero such that $\delta_n{\bar K}\to+\infty$ and
  for any non-random graph sequence $\{R_n\}$ with $V(R_n)=V(G(n))$
and $e(R_n) \le r_n$
\begin{equation}\label{2013_08_16.lem.plantedrig}
\pr \left( Rainbow(G(n), R_n, {\bar K}(1+\delta_n)) \right) \le \eps_n, 
\qquad 
n = 1, 2, \dots
\end{equation}
\end{lemma}
Here we choose $\{\delta_n\}$ such that ${\bar K}(1+\delta_n)$ were an integer.

\medskip

\begin{proof}
Let $\{x_n\}$ be a positive sequence such that
 \begin{equation*} \label{x-2013-08-14}
   p x_n \to 0, \quad  x_n\ll mp 
\quad
\mbox{ and } 
\quad
\sqrt{mp \ln n} \ll x_n
 \end{equation*}
(one can take, e.g., $x_n=\varphi_n\sqrt{mp\ln n}$, with $\varphi_n\uparrow +\infty$ 
satisfying
$\varphi^2_n {\bar K}p\to 0$).

Given $n$, we truncate the random sets $S_v$, 
prescribed to vertices $v\in V$ of the graph $G=G(n,m,p)$, 
to the size
 $M = \lfloor mp + x_n\rfloor$. Denote 
\begin{equation*}
  \bar{S}(v)=
  \begin{cases}
    S_v, \text{ if } |S_v| \le M,\\
    \text{$M$ element random subset of $S_v$, otherwise}.
  \end{cases}
\end{equation*}
We remark that for 
the event $B=\{S_v={\bar S}_v, \forall v\in V\}$ Chernoff's bound  implies
\begin{equation}\label{2013_08_16.B}
\pr(B)=1-o(1). 
\end{equation}

Now, let  $t \in [K; 2 K]$ and  let $T = \{u_1, \dots, u_t\}$ be a subset of $V$
of size $t$. 
By $R_T$ we denote the subgraph of $R_n$ induced by the vertex set $T$.
Given $i \in \{1, \dots, t\}$, let $T_i\subseteq \{u_1,\dots u_{i-1}\}$ denote the subset of 
vertices which are not adjacent to $v_i$ in $R_n$. Let $A_T(i)$ denote the event that
   sets
$\{\bar{S}_{u} \cap S_{u_i}, \, u\in T_i\}$
have distinct representatives (in particular, none of the sets is empty).
Furthermore, let $A_T$ denote the event that all $A_T(i)$, $1\le i \le t$ hold simultaneously
\[
A_T = \bigcap_{i=1}^t A_T(i).
\]
We shall prove  below that 
whenever $n$ is large enough
\begin{equation}\label{eq.AT}
  \pr(A_T) \le \left( 1- (1-p)^M \right)^{\binom t 2 - e(R_T)}.
\end{equation}
Next, proceeding as in  Lemma~\ref{lem.planted} we find positive sequences 
$\{\delta'_n\}$, $\{\epsilon'_n\}$ converging to zero such that the number $X'_n$ of 
subsets $T\subseteq V$ of size 
\begin{displaymath}
 a'= \Big\lfloor \frac{2\ln n}{(1-p)^M}(1+\delta'_n)\Big\rfloor
\end{displaymath}
that satisfy the event $A_T$ has expected value $\E X'_n\le \epsilon'_n$. 
For this purpose, we 
apply (\ref{2013-07-27}) to $a'$ and $p'=1-(1-p)^M$, and use (\ref{eq.AT}). We remark that 
$a'={\bar K}(1+\delta''_n)$, where $\{\delta''_n\}$ converges 
to zero  and $\delta''{\bar K}\to+\infty$. Indeed, we have 
$\delta_n'\ln n/(1-p)^M\to+\infty$, by Lemma~\ref{lem.planted}, and we have 
$(1-p)^M=e^{-mp^2 - O(px + mp^3)}$ 
with
$px+mp^3=o(1)$. In particular, for large $n$, we have 
$a'\in [{\bar K}, 2{\bar K}]$.   

The key observation of the proof is that events $B$ and $Rainbow(G, R_n, a' )$ 
imply  
$X'_n > 0$. Hence,
\begin{displaymath}
 \pr(Rainbow(G, R_n, a' )\cap B)\le \pr(X'_n>0)
\le 
\E X'_n
\le 
\epsilon'_n.
\end{displaymath}
In the last step we used Markov's inequality.
Finally, invoking (\ref{2013_08_16.B}) we obtain 
(\ref{2013_08_16.lem.plantedrig}).

\vskip 0.2cm

It remains to show (\ref{eq.AT}). We write
\[
\pr(A_T) = \prod_{i=1}^t \pr \left(A_T(i) | A_T(1), \dots, A_T(i-1)\right)
\]
and evaluate, for $1\le i\le t$,
\begin{equation}\label{2013_08_17.AT}
\pr(A_T(i) | A_T(1), \dots, A_T(i-1)) \le \left( 1 - (1-p)^M \right)^{|T_i|}.  
\end{equation}
Now (\ref{eq.AT}) follows from the simple  identity 
$\sum_{1\le i\le t}|T_i|={\binom t 2}-e(R_T)$.
Let us prove (\ref{2013_08_17.AT}). For this purpose we apply Lemma~\ref{lem.DR}.
We first
condition on $\{{\bar S}_{u}$, $u\in T_i\}$ and the size
$|S_{v_i}|$ of $S_{v_i}$. By 
Lemma~\ref{lem.DR} the conditional probability 
\[
\pr(A_T(i)\,  \bigl| \, \bar{S}_{u},\ u\in T_i, \, |S_{v_i}|)
\]
is maximized when the sets ${\bar S}_{u}$, $u\in T_i$ are mutually disjoint
(at this step we check the condition of  
Lemma~\ref{lem.DR} that
$\sum_{u\in T_i}|{\bar S}_{u}|\le tM<m$, for large $n$). Secondly, we drop
the conditioning on $|S_{v_i}|$ and allow $S_{v_i}$ to choose its element independently at 
random with probability $p$. In this way we obtain (\ref{2013_08_17.AT}).
\end{proof}

\begin{lemma}\label{lem.G1general}
Let $\{G(n)\}$ be a sequence of random binomial intersection graphs, 
where $m=m(n)\to+\infty$
and $p=p(n)\to 0$ as $n\to+\infty$.
 Assume that 
\begin{displaymath}
np = O(1),
\qquad
m (n p)^3 \ll  {\bar K}^2,
\end{displaymath}
where ${\bar K} = 2 e^{mp^2}\ln n$. 
Assume, in addition, that (\ref{2013_08_17.mpK}) holds.  

  Then there is a sequence $\{\delta_n\}$ converging to zero such that 
$\delta_n{\bar K}\to+\infty$ and
  \[
   \pr \left(\omega(G(n)) > {\bar K}(1 + \delta_n) \right) \to 0.
   \]
\end{lemma}

\begin{proof} Given $n$, let $U$ be a random subset of $V=V(G(n))$ 
with binomial number of elements $|U|\sim Bin(n,p)$ and such that, 
for any $k=0,1,\dots$,  conditionally,
given the event $|U| = k$, 
the subset $U$ is uniformly 
distributed 
over the class of subsets of $V$ of size $k$.
  Recall that $T_w\subseteq V$ denotes the set of vertices that have chosen an attribute 
  $w \in W$.  
We remark that  $T_w$, $w \in W$ are iid random subsets 
having the same probability distribution as $U$.

We call an attribute $w$ \emph{big} if $|T_w| \ge 3$, otherwise $w$ is \emph{small}. Let $W_B$ and $W_S$
denote the  sets of big and small attributes. Denote by $G_{B}$ 
(respectively, $G_S$)
the subgraph of $G=G(n)$ consisting
of edges  
covered by big (respectively, small) attributes.
We observe that, given $G_B$, the random sets $T_z$, $z\in W_S$, defining the edges of $G_S$  are (conditionally) independent.
We are going to replace them by bigger sets, denoted $T'_z$, by adding some more elements as follows. 
Given $T_z$, we first generate independent random variables 
${\mathbb I}_z$ and 
$|\Delta_z|$, where ${\mathbb I}_z$ has Bernoulli distribution with success probability $p'=\pr(|U|\le 2)$ and 
where $\pr(|\Delta_z|=k)=\pr(|U|=k)/(1-p')$, $k=3,4,\dots$. 
Secondly, for ${\mathbb I}_z=1$ we put $T'_z=T_z$. Otherwise we put $T'_z=T_z\cup \Delta_z$, where
$\Delta_z$ is a  subset of $V\setminus T_z$ of size $|\Delta_z|-|T_z|\ge 1$ drawn uniformly at random. 
We note that given $G_B$, the random sets $T'_z$, $z\in W_S$ are (conditionally) independent and have the same probability distribution as $U$.
Next we generate independent random subsets $T'_w$ of $V$, for $w\in W_B$, so that they have the same distribution as $U$ and were independent 
of $G_S$, $G_B$ and $T'_z$, $z\in W_S$. Given $G_B$, the collection of random
sets $\{T'_w, w\in W_B\cup W_S\}$ defines the binomial random intersection graph $G'$ having the same distribution as $G(n,m,p)$.

We remark that $G_S\subseteq G'$ and every edge of $G_S$ can be assigned a unique small 
attribute that covers this edge 
and the assigned attributes are all different. 
On the other hand, the graph $G_B$ is relatively small.
Indeed,
since each  $w$ covers $\tbinom{|T_w|}{2}$ edges, 
   the expected number of edges  of $G_{B}$ is at most
\begin{eqnarray}\nonumber
 \E \sum_{w\in W}\binom{T_w}{2}{\mathbb I}_{\{|T_w|\ge 3\}}
=
m\E\binom{T_w}{2}{\mathbb I}_{\{|T_w|\ge 3\}}
\le 
m\sum_{k\ge 3} \binom k 2 \binom n k p^k.
\end{eqnarray}
Invoking the simple bound
   \[
  \sum_{k\ge 3} \binom k 2 \binom n k p^k \le  (np)^2 (e^{np} - 1)/2  = O((np)^3)
  \]
we obtain $\E e(G_B)=O(m(np)^3)$.

Now we choose an integer sequence $\{r_n\}$ such that  
 $m (n p)^3 \ll r_n \ll {\bar K}^2$ and write, for an integer $K'>0$,
 \begin{equation}\label{2013_08_19}
 \pr\left(\omega(G)\ge K'\right)
 \le 
 \E \pr\left(\omega(G)\ge K'|G_B\right){\mathbb I}_{\{e(G_B)\le r_n\}}+
\pr\left(e(G_B)\ge r_n\right).
 \end{equation}
 Here, by Markov's inequality,  $\pr(e(G_B)\ge r_n)\le r_n^{-1}\E e(G_B)=o(1)$. 
Furthermore, we 
 observe that $\omega(G)\ge K'$ 
implies the event $Rainbow(G',G_B, K')$. Hence,
\begin{displaymath}
 \pr\left(\omega(G)\ge K'|G_B\right)
\le
\pr\left(Rainbow(G',G_B, K')|G_B\right).
\end{displaymath}
We choose $K'={\bar K}(1+\delta_n)$
 and apply Lemma~\ref{lem.plantedrig} 
 to the conditional  probability on the right. 
At this point we specify $\{\delta_n\}$ and find $\epsilon_n\downarrow 0$
 such that $\pr\left(Rainbow(G',G_B, K')|G_B\right)\le \epsilon_n$ 
uniformly in $G_B$ satisfying 
$e(G_B)\le r_n$.
 Hence, (\ref{2013_08_19}) implies 
$\pr\left(\omega(G)\ge {\bar K}(1+\delta_n)\right)\le \epsilon_n+o(1)=o(1)$.
\end{proof}

\medskip

Now we are ready to prove Lemma~\ref{lem.G1}.

\medskip

\begin{proofof}{Lemma~\ref{lem.G1}}
Let
\begin{equation}\label{2013_08_21.eps}
0<\epsilon<2^{-1}\min\{1,\, 1-2^{-1}\alpha,\, \beta-2+\alpha-6\alpha\epsilon_1\}
\end{equation}
and let ${\bar G}_1$ be the subgraph of $G_1$ induced by vertices $v\in V_1$ with
$X_v\le \theta$. Here
$\theta^2=(1-\varepsilon-2^{-1}\alpha)m\ln n$. Let 
$D=|V(G_1)\setminus V({\bar G}_1)|$ denote the number of vertices of $G_1$ that
do not belong to ${\bar G}_1$. 

To prove the lemma 
we write 
$\omega(G_1)\le D+\omega({\bar G}_1)$ and show that each summand on the right is of order
 $o_P(K)$ for appropriately chosen $\epsilon=\epsilon(n)\to 0$.

Using (\ref{eq.pldef2}) and Lemma~\ref{lem.sv} we estimate the expected value of $D$ 
for $n\to+\infty$ 
\begin{equation}\label{2013_08_21+}
\E D=n\left(\pr(X_v\ge \theta)-\pr(X_v\ge \theta_2)\right)
\le (h(\epsilon)+o(1))K.
\end{equation}
Here 
$h(\epsilon):=(1-\epsilon-2^{-1}\alpha)^{-\alpha/2} - (1-2^{-1}\alpha)^{-\alpha/2}\to 0$
as $\varepsilon\to 0$.
Letting $\epsilon\to 0$ we obtain from (\ref{2013_08_21+}) that $D=o_P(K)$. 

We complete the proof  by showing that for any $\varepsilon$ satisfying 
(\ref{2013_08_21.eps})
\begin{equation}\label{2013_08_22.++}
 \pr\left(\omega({\bar G}_1) \ge 4 n ^ {1 - 2^{-1}\eps-2^{-1}\alpha} \ln n \right)=o(1).
\end{equation}
 Note that 
$n ^ {1 - 2^{-1}\eps-2^{-1}\alpha} \ln n \ll K$.

Let ${\bar N}$ be a binomial  random variable,  ${\bar N}\sim Bin(n,\pr(X_v>\theta_1))$,
and
let 
\begin{displaymath}
{\bar n}=(1+\epsilon)n^{1- 2^{-1}\alpha + \alpha \eps_1}L(n^{0.5-\epsilon_1})
\quad
\
{\text{ and}}
\quad
\
{\bar p}^2=(1-2^{-1}\epsilon-2^{-1}\alpha)m^{-1}\ln n. 
\end{displaymath}
We couple ${\bar G}_1$ with 
the binomial random intersection graph $G'=G({\bar n}, m,{\bar p})$ so that
the event that ${\bar G}_1$ is isomorphic to a subgraph of $G'$, denoted 
${\bar G}_1\subseteq G'$, has probability 
\begin{equation}\label{2013_08_22.couple}
\pr({\bar G}_1\subseteq G')=1-o(1). 
\end{equation}
We argue that  such a coupling is possible because the events
$A=\{$every vertex of $G'$ is prescribed at least $\theta$ attributes$\}$
and $B=\{|V({\bar G}_1)|\le {\bar n}\}$ have very high probabilities. Indeed,
the bound $\pr(A)=1-o(1)$ follows from 
Chernoff's inequality (\ref{eq.barB}). To get the bound $\pr(B)=1-o(1)$ we first couple
binomial random variables $|V({\bar G}_1)|\sim Bin(n,\pr(\theta_1<X_v<\theta))$ 
and ${\bar N}$
so that $\pr(|V({\bar G}_1)|\le {\bar N})=1$ and then invoke the bound
$\pr({\bar N}\le {\bar n})=1-o(1)$, which follows from 
Chernoff's inequality.

Next we apply Lemma~\ref{lem.G1general} to $G'$ and obtain the bound
\begin{equation}
   \pr \left( \omega(G') > 4 n ^ {1 - 2^{-1}\eps-2^{-1}\alpha} \ln {\bar n }\right)=o(1),
\end{equation}
which together with (\ref{2013_08_22.couple}) implies (\ref{2013_08_22.++}).
\end{proofof}

\section{Finite variance}
\label{sec.finite.variance}

In this section we prove Theorem~\ref{thm.finite}.
We note that the random power-law graph studied by Janson, \L uczak and Norros \cite{jln10} 
whp does not contain $\ck_4$ as a subgraph if the degree distribution has a 
finite second moment.
In our case a similar result holds for the rainbow $\ck_4$. 
Given a sequence of 
random intersection 
graphs
$\{G(n)\}$, we show that the number of rainbow
$\ck_4$ subgraphs of $G(n)$ 
is stochastically bounded as $n\to+\infty$ provided that the sequence of the 
second moments of the
degree distributions is bounded. If, in addition, the sequence of degree distributions 
is uniformly 
square integrable, then $G(n)$ has no rainbow $\ck_4$ whp, see Lemma~\ref{lem.noRK4} below. We use these observations in the 
proof of Theorem~\ref{thm.finite}.

\subsection{Large cliques and rainbow $\ck_4$}
\label{subsec.finite.structure}

Let $U$ be a finite set and let $\cc=\{C_1,\dots, C_r\}$ be a collection of 
(not necessarily distinct) subsets of $U$. We consider the complete graph $\ck_U$ 
on the vertex set 
$U$ and interpret subsets $C_i$ as colours: an edge $x\sim y$ receives colour $C_i$ 
(or just $i$) whenever $\{x,y\}\subseteq C_i$. We call $\cc$ a {\it clique cover} if
every edge of the clique $\ck_U$ receives at least one colour. 
The edges spanned by the vertex set
$C_i$ form a subclique, which we call the {\it monochromatic clique} of colour $i$.
We say that a vertex set $S\subseteq U$ is a {\it witness of a rainbow clique}
if every edge of the clique $\ck_S$ induced by $S$ receives a non-empty collection 
of colours
and it is possible to assign each edge one of its colours so 
 that all edges of $\ck_S$ were assigned different colours.
For example, the collection $\cc=\{A,B,C\}$, where $A=\{1,2,3\}$, $B=\{1,3, 4\}$ and 
$C=\{2,4,3\}$ is a clique cover of the set $\{1,2,3,4\}$. It produces three monochromatic triangles
and four rainbow triangles. 

We start with a result that relates clique covers to rainbow clique subgraphs. For a clique cover
$\cc = \{C_1, \dots, C_r\}$ denote by $p(\cc) = \max_{i\ne j} |C_i \cap C_j|$ the size of maximum
pairwise intersection.
\begin{lemma}\label{lem.bound4}
Let $k$ and $p$  be positive integers. Let $h=h(k)>0$ denote
 the smallest integer such that $\binom{h}{4}\ge k$. 
 Let $\cc = \{C_1, \dots, C_r\}$ be a clique cover of a finite set $U$
 and assume that $\max_{C \in \cc} |C| \ge |U| - h$ and  $p(\cc) \le p$.
   
     If, in addition, $|U|\ge t(k,p)$, where 
$t(k,p)=c\frac {h^3}{\ln h} p\left (\sqrt{ 2 k} + 5 + 2p \right)$, 
then $\cc$ produces  at least $k$ witnesses of rainbow $\ck_4$.
     Here $c$ is the  absolute constant of Lemma~\ref{lem.rainbow}.
\end{lemma}

\begin{proof} \,
Write $b = \max_i |C_i|$. 
We note that $\cc$ 
has no rainbow $\ck_h$ since otherwise there would be at least $\binom{h}{4}\ge k$ copies
of rainbow $\ck_4$. Observe, that every monochromatic subclique of $\ck_U$ has at most 
$b$
vertices. Hence, each colour appears at most $b-1$ times at each vertex of $\ck_U$. By 
Lemma~\ref{lem.rainbow}, $\ck_U$ has at most $c(b-1)h^3/\ln h$ vertices. That is, 
$b>a|U|$, where $a = \frac {\ln h} {c h^3}$ and 
$c$ is an absolute constant. Fix $B\in \cc$ with $|B|=b$ and a subset $S\subseteq U\setminus B$ of size $h$, say $S=\{x_1,\dots, x_h\}$. 
Here we use the assumption $|U|\ge b+h$ telling that $U\setminus B$ has at least $h$ elements, $|U\setminus B|=|U|-b\ge h$.
We remark, that at least one pair of vertices of $S$, say $\{x_1, x_2\}$, receives at most $5$ colours (it 
is covered by at most $5$ sets from $\cc$). Indeed, otherwise every edge of $\ck_S$ received at least $6$ distinct colours and, thus, each $S'\subseteq S$ of size $|S'|=4$ induced
a rainbow $\ck_4$. This contradicts to our assumption that there are fewer than  $k\le \binom{h}{4}$ rainbow copies of $\ck_4$.

 We observe that the set of colours received by the pair $\{x_1, x_2\}$ is non-empty (since $\cc$ is a clique cover) 
 and fix one such colour, say
 $C_{x_1,x_2}\in \cc$. Now, consider the set of pairs $\{\{x_1,y\},\, y\in B\}$ and pick a smallest family of sets from $\cc$ 
 such that each pair were covered by a member of the family 
 (the smallest family means that any other family with fewer members would leave at least one uncovered pair). Since each member of the family intersects with $B$ in at most $p$ vertices (condition of the lemma) we conclude that such a family contains at least $\lceil b/p\rceil$ members. Furthermore, since the family is minimal, every member covers a pair $\{x_1,y\}$ which is not covered by other members. Hence, we can 
 pick a set $B_1\subseteq B$ 
 of size $\lceil b/p\rceil$ so that every $\{x_1,y\}$, $y\in B_1$ is covered by a unique member, say $C_{x_1,y}$, of the family.

    Next, remove from $B_1$ the elements $y$ such that $x_2 \in C_{x_1,y}$ (there are at most $5$ of them).
    Then remove those elements $y$ which belong to the set $C_{x_1, x_2}$ (there are at most $p$ of them, since 
    $|C_{x_1,x_2}\cap B|\le p$).
    Call the newly formed set $B'$. Notice that
    \[
    b':= |B'| \ge \frac {b} p - 5 - p > \frac {a |U|} p - 5 - p. 
    \]
    Let us consider the clique ${\tilde  K}$ on the vertex set $B' \cup \{x_1, x_2\}$.
    For $y \in B'$, colour each edge $\{x_1, y\}$ of ${\tilde K}$ with the colour $C_{x_1,y}$. Colour the edge $\{x_1,x_2\}$
    with $C_{x_1, x_2}$ and for every edge $\{y_i, y_j\} \in B'$ use the colour $B$. 
Finally, for $y \in B'$, assign $\{x_2, y\}$ an arbitrary colour from the set of colours received by $\{x_2, y\}$  from the clique cover $\cc$. 
 
    We claim that for any  $y_1 \in B'$ and any $y_2 \in B' \setminus C_{x_2, y_1}$,
    the set $\{x_1, x_2, y_1, y_2\}$ witnesses a rainbow $\ck_4$. Indeed, by the construction,
    the colour $C_{x_1, x_2}$ of the edge $\{x_1,x_2\}$ occurs only once, because $B' \cap C_{x_1, x_2} = \emptyset$.
    Similarly, for $x_1, x_2 \not \in B$, the colour $B$ of $\{y_1,y_2\}$ occurs only once. 
    The colours of the two other edges incident to $x_1$ occur only once, since we removed all candidates $y$ such that $x_2  \in C_{i_{x_1,y}}$,
    while constructing the set $B'$.
    Finally, we have $C_{x_2, y_1} \ne C_{x_2, y_2}$ since we chose $y_2$
    outside $C_{x_2, y_1}$.

    How many such witnesses can we form? For any $y_1$ we choose 
$|B'| - |B' \cap C_{x_2,y_1}| \ge |B'| - p$ suitable $y_2$.
    Repeating this for each $y_1$ we will produce every $4$-set 
    at most
    twice. Therefore  ${\tilde K}$ contains at least
    \begin{equation}\label{2013_08_27}
    \frac {b'(b' - p)} 2 \ge \frac 1 2 \left( \frac {a |U|} p - 5 - 2 p \right)^2
    \end{equation}
    witnesses of rainbow $\ck_4$. But since the total number of witnesses of rainbow $\ck_4$ produced by $\cc$ is less that $k$, the right-hand side of (\ref{2013_08_27}) is less than $k$. We obtain the inequality
    \[
    |U| < \frac p a \left (\sqrt{ 2 k} + 5 + 2p \right)=t(k,p),
    \]
    which contradicts to the condition $|U|\ge t(k,p)$.
\end{proof}

\medskip
In the remaining part of the subsection \ref{subsec.finite.structure} 
we interpret  attributes $w\in W$ 
as colours assigned to edges of
a random intersection graph.
\begin{lemma}\label{lem.PRK4}
    Let $G = G(k ,m, P)$ be a random intersection graph and let
    $X_1, \dots, X_k$ denote the sizes of random sets defining  $G$. For any integers
    $x_1, \dots, x_k$ such that the event 
    $B=\{X_1 = x_1, \dots, X_k = x_k\}$ has positive probability, we have
    \[
    \pr(G \mbox { has a rainbow } \ck_k | B) \le m^{-\frac{k(k-1)} 2} (x_1 x_2 \dots x_k)^{k-1}.
    \]
\end{lemma}
\begin{proof}\,
Our intersection graph produces a rainbow clique on its $k$ vertices whenever for some injective mapping, say $f$,  from the set of pairs of vertices to the set of attributes, the event $A_f=\{$every pair $\{x,y\}$ is covered by $f(\{x,y\})\}$ occurs. 
    By the independence,  $\pr(A_f|B)=\prod_i\frac{(x_i)_{k-1}}{(m)_{k-1}}$. Since there are $(m)_{\binom k 2}$ possibilities to choose the map $f$,
we obtain, by the union bound,
  \begin{displaymath}
  \pr(G \mbox { has a rainbow } \ck_k | B)
  \le 
  (m)_{\binom k 2}  \prod_i\frac{(x_i)_{k-1}}{(m)_{k-1}}
  \le 
  \frac {(x_1 x_2 \dots x_k)^{k-1}} {m^{k(k-1)/2}}. 
\end{displaymath}
\end{proof}

\begin{lemma}\label{lem.noRK4}
    Let $\{G(n)\}$ be a sequence of
    random intersection graphs such that $\E Y(n) ^2  = O(1)$. 
    Then the number $R = R(n)$  of 4-sets $S \subseteq V(G(n))$ that 
witness a rainbow $\ck_4$ in $G(n)$ satisfies as $n\to+\infty$
    \[
     \E R \le \frac { (\E Y^2)^4} {4!} = O(1).
    \]
    Furthermore, if for some positive sequence $\eps_n \to 0$ 
we have
 $n \pr( Y(n) \ge \eps_n n^{1/2}) \to 0$
    then $G(n)$ does not contain a rainbow $\ck_4$ whp.

\end{lemma}

\begin{proofof}{Lemma~\ref{lem.noRK4}}
    Denote $X_v =  |S_v(n)|$ and  $Y = Y(n)$.
We write, using symmetry and the bound of Lemma~\ref{lem.PRK4},
    \begin{align*}
        &\E R 
=\sum_{S\subseteq V, |S|=4}\pr(S \, {\text{witnesses a rainbow}} \, \ck_4)
\le 
\binom n 4 \E \left(\frac{ (X_1 X_2 X_3 X_4)^3} {m^6} \wedge 1 \right).
    \end{align*}
Next, we apply the simple  inequality $a^6\wedge 1\le a^4$ and bound 
the right-hand side from above by
$\frac {n^4} {4!} \frac {\E (X_1 X_2 X_3 X_4)^2} {m^4} = \frac { (\E Y^2) ^4} {4!}$.

    For the second part of the lemma, 
let $b = b(n) =  \eps_n \sqrt{m}$ and let $A = A(n)$ be the event that 
    $\max_{v \in V} X_v \le b$. Let ${\bar A}$ denote the complement event.
We write
    \begin{equation}
        \pr(R \ge 1) 
\le
\pr(R\ge 1, A)+\pr({\bar A}) 
        \le \E R \I_A + \pr({\bar A}). 
    \end{equation}
    By the union bound the second term is at most
    \[
    n \pr(X > b)  = n \pr(Y > \eps_n n^{1/2}) \to 0.
    \]
    The first term by Lemma~\ref{lem.PRK4} satisfies
    \begin{displaymath}
\E R \I_A 
\le 
\binom n 4 m^{-6} \E (X_1 X_2 X_3 X_4)^3 \I_A 
\le 
\frac {(\E X^2) ^ 4 n ^ 4 b ^ 4} {4! m^6}
=(\eps_n \E Y^2)^4
=o(1).
    \end{displaymath}
\end{proofof}

\medskip

The next result shows that the structure of 
random intersection graphs with $\E Y(n)^2 = O(1)$ is relatively simple.
\begin{lemma}\label{lem.intersection}
Let $\{G(n)\}$ be a sequence of 
    random intersection graphs. Assume that $\E Y(n)^2 = O(1)$ and
    $m(n) \to \infty$  as $n\to+\infty$. Then whp each pair $\{w', w''\}$ of attributes 
    is shared by at most two vertices of $G(n)$.
\end{lemma}

The lemma says that the intersection of any two monochromatic cliques of $G(n)$ 
consists of at most one edge whp.
\medskip

\begin{proof}
For any pair of attributes $w', w''$ and a vertex $v$ of $G(n)$, we have
    \begin{align*}
        \pr(w', w'' \in S_v) &= \sum_{k=0}^m \pr(|S_v| = k) \frac{k(k-1)} {m(m-1)} = \frac {\E X^2 -\E X} {m(m-1)} 
        \\&\le  \frac {\E Y^2 } {n(m-1)} \le \frac c {n m}.
    \end{align*}
Here  $c >0$ does not depend on $m$ and $n$.
   By the union bound, the probability that there is 
a pair of attributes
shared by  $k$ or more vertices 
is at most
    \begin{align*}
        \binom m 2 \binom n k \pr(w',w'' \in S_v)^k &\le m^2 \left( \frac {e n} k \right)^k \left( \frac c {n m} \right)^k
        \le m^2\left( \frac {e c} {k m} \right)^k.
    \end{align*}
This probability tends to zero for any $k\ge 3$.
\end{proof}

\vskip 0.5 cm
\begin{proofof}{Theorem~\ref{thm.finite}} \,
    Let $R = R(n)$ denote the number of 4-sets $S \subseteq V(G(n))$ 
witnessing rainbow $\ck_4$ in $G(n)$.
    By Lemma~\ref{lem.intersection}, the intersection of any two monochromatic cliques 
has at most 2 vertices whp. 
    In that case, by Lemma~\ref{lem.bound4} 
(applied to the set of vertices $U$ of the largest clique)
 either $\omega(G(n)) < t(R+1, 2)$ or 
 $\omega(G) \le \omega'(G) + h(R+1)$. 
%
    Thus, 
    \[
      \omega(G(n)) \le  \omega'(G(n)) + Z(n)
    \]
    where $Z(n) = t(R+1, 2) + h(R+1)=O_P(1)$, by Lemma \ref{lem.noRK4}.
%
%
%

    If $n \pr(Y(n) > \eps_n n^{1/2}) \to 0$ for some $\eps_n \to 0$ then by Lemma~\ref{lem.noRK4}
    $G(n)$ whp does not contain a rainbow $\ck_4$, so whp $\omega(G) \le t(1,2) \vee (\omega'(G) + 3)$.
\end{proofof}

\subsection{Monochromatic cliques and balls and bins}
\label{subsec.finite.bb}

Here we prove Theorem~\ref{thm.omega.mono}. In the proof we use the fact that
the maximum bin load 
 $M(N, m)$ is  a ``smooth'' function of the first argument $N$, see lemma below.

\begin{lemma}\label{lem.epsbb}
    Let $\{N_n\}$ and $\{m_n\}$ be sequences of positive integers such that 
$N = N_n \to \infty$ and $m = m_n \to \infty$.
Let $\{\delta_n\}, \{\eps_n\}$ be positive sequences converging to zero
such that $\eps_n = o(\delta_n)$.
For every $n$ there is a coupling between random variables
$M' = M'_n = M( \left \lfloor N(1 + \eps_n) \right \rfloor, m)$ and $M = M_n = M(N,m)$
 such that $M \le M'$ with probability one,
and 
\begin{equation}\label{eq.bbcoupling1}
       \pr(M' - \delta_n \E M' \le M) \to 1.
   \end{equation}
   If, additionally, $\pr (M' > \delta_n^{-1}) \to 0$, then  $M = M'$ whp. 
\end{lemma} 
\begin{proof} \, Given $n$, we label $m$ bins by numbers $1,\dots, m$.
  Throw $\left \lfloor N (1 + \eps_n) \right \rfloor$ balls into
  bins. This gives an instance of $M'$. 
Denote by $L$ the label of the bin with the lowest index
  realising the maximum. 

  Now delete  uniformly at random  $\left \lfloor \eps_n N \right \rfloor$ balls. 
The configuration
  with the remaining $N$ balls gives an instance of $M\le M'$. 
We remark that conditionally, given $M'$, the number $\Delta$ 
of balls deleted from the bin $L$
has a hypergeometric distribution with the  mean value
\begin{displaymath}
\frac
{M'\times \left \lfloor \eps_n N \right \rfloor}
{\left \lfloor N (1 + \eps_n) \right \rfloor}\le \eps_n M'. 
\end{displaymath}
Now the bin $L$ contains  $M'-\Delta\le M$ balls and, by Markov's inequality,
\begin{displaymath}
 \pr(M'-M\ge t)\le \pr(\Delta\ge t)\le t^{-1}\E\Delta\le t^{-1}\eps_n\E M'.
\end{displaymath}
Choosing $t=\delta_n\E M'$ yields (\ref{eq.bbcoupling1}). Similarly, if
$\pr (M' \ge \delta_n^{-1}) = o(1)$, then 
\[
\pr(M' - M \ge 1) \le \E \Delta \ii_{M' \le \delta_n^{-1}} + \pr (M' > \delta_n^{-1}) \le \eps_n \delta_n^{-1} + o(1) \to 0. 
\]
\end{proof}

\medskip

\begin{proofof}{Remark~\ref{rmk.omega.mono}}
  Suppose $m = o(n)$, $\E Y = \Theta(1)$ and $\E Y^2 = O(1)$.
  Since $X=X(n)$ is a non-negative integer,
  we have $\E X^2 \ge \E X$. But $\E X^2 = O(m/n)$ and $\E X = \Theta( (m/n)^{1/2})$, so
  $\E X^2 = o(\E X)$, a contradiction.
\end{proofof}

\bigskip

\begin{proofof}{Theorem~\ref{thm.omega.mono}} 
 In view of Remark~\ref{rmk.omega.mono} it suffices to consider the case
$m = \Omega(n)$. Denote $\eps_n=(2+\ln^2n)^{-1}$ so that $\eps_n\ln n=o(1)$ and 
$n\eps_n^2 \to+\infty$. Given $n$, write $\eps=\eps_n$ and   denote 
${\bar N} =n\E X_1= \sqrt{mn}\, \E Y$ and 
\begin{displaymath}
{\bar N}^{-}=\lfloor {\bar N}(1-4\eps)\rfloor,
\qquad
{\bar N}^{+}=\lceil {\bar N}(1+4\eps)\rceil.
\end{displaymath}
%
In order to generate an instance of $G(n)$ 
we  draw a random sample $X_1,\dots, X_n$
from the distribution $P(n)$. Then choose random subsets $S_{v_i}\subseteq W$ of size
$X_i$, $v_i\in V$, by  throwing  balls
into $m$ bins labelled 
$w_1,\dots, w_m$ (the $j$-th bin has label $w_j$ and index $j$) as follows. 
Keep throwing balls labelled $i=1$ until there are exactly $X_i$ different bins
containing a ball labelled $i$. 
Do the same for $i = 2, \dots, n$. Now, for each $i$, the bins 
containing  balls  labelled $i$ 
make up the set $S_{v_i}$. In this way we obtain an instance of $G(n)$.
Let $X_i'$ denote the number of balls of label $i$ thrown so far. 
Clearly, $X_1', \dots, X_n'$ is a sequence 
of independent random variables and $X_i' \ge X_i$, for each $i$.
We stop throwing balls if the number of balls 
$N'=\sum_{i}X'_i$ at least as large as ${\bar N}^{+}$. Otherwise we throw additional 
${\bar N}^{+}-N'$ unlabelled balls into bins. 

Let us inspect the bins after $j$ balls have been thrown. Let
 ${\mathcal M}(j)$ denote the set of balls contained in the bin with the largest
number of balls and the smallest index. We note that the number  $M(j) = |\mathcal{M}(j)|$
of balls in that bin has the same distribution as $M(j,m)$ (random variable defined before
Theorem~\ref{thm.omega.mono}). 

Denote, for short, $\omega'=\omega'(G(n))$ and 
${\bar M}=M( \lfloor{\bar N} \rfloor)$.
We observe that the event 
${\cal A}_1=\{$all balls of $\mathcal{M}(N')$ have different labels$\}$ implies 
$\omega'(G(n))=M(N')$.
Furthermore, if both  events ${\cal A}_2=\{M({\bar N}^{-})=M({\bar N}^{+})\}$ and 
${\cal A}_3=\{{\bar N}^{-}\le N'\le {\bar N}^{+}\}$ hold, then ${\bar M}=M(N')$.
We shall show below that 
\begin{equation}\label{2013_08_30}
 \pr({\cal A}_r)=1-o(1),
\qquad
{\text{for}}
\qquad
r=1,2,3.
\end{equation}
Now, (\ref{2013_08_30}) implies $\pr(\omega'={\bar M})=1-o(1)$ and, since the distributions of $M(\lfloor {\bar N} \rfloor, m)$ and ${\bar M}$ coincide, we obtain
\begin{displaymath}
 d_{TV}\bigl(\omega', M(\lfloor{\bar N}\rfloor, m)\bigr)
=
d_{TV}\bigl(\omega', {\bar M}\bigr)
\le 
\pr\bigl(\omega'\not= {\bar M}\bigr)
=o(1).
\end{displaymath}
It remains to prove (\ref{2013_08_30}). Let us  consider $\pr ({\cal A}_3)$.
We first replace $X_i$ and $X_i'$ by the truncated random variables
\begin{displaymath}
{\tilde X}_i=X_i{\mathbb I}_{\{X_i\le \eps \, m\}}
\qquad
{\text{and}}
\qquad
{\tilde X}_i'=X_i'{\mathbb I}_{\{X_i\le \eps \, m\}},
\quad 
1\le i\le n.
\end{displaymath}
Denote ${\tilde N}'=\sum_{i}{\tilde X}_i'$ and introduce events
${\tilde {\cal A}}_3=\{{\bar N}^{-}\le {\tilde N}'\le {\bar N}^{+}\}$ and
${\cal A}_4=\{\max_{1\le i\le n} X_i\le \eps \,m\}$.   Let 
${\bar {\cal A}}_4$ denote the complement of ${\cal A}_4$.     From the relation
${\cal A}_3\cap {\cal A}_4={\tilde {\cal A}}_3\cap {\cal A}_4$ 
we obtain
\begin{displaymath}
 \pr({\cal A}_3)
\ge 
\pr({\cal A}_3\cap {\cal A}_4)
=
\pr({\tilde {\cal A}}_3\cap {\cal A}_4)
\ge 
\pr({\tilde {\cal A}}_3)-\pr({\bar {\cal A}}_4).
\end{displaymath}
Furthermore, by the union bound and Markov's inequality,   we have 
\begin{displaymath}
 \pr({\bar {\cal A}}_4)
\le 
n\pr(X_1>\eps \, m)\le n\frac{\E X_1^2}{\eps^2m^2}=
\frac{\E Y^2}{\eps^2m}=o(1),
\end{displaymath}
since $m=\Omega(n)$ and $\eps^2n\to+\infty$. 
Hence, $\pr({\cal A}_3)\ge \pr({\tilde {\cal A}}_3)-o(1)$. Secondly,  we prove that
$\pr({\tilde {\cal A}}_3)=1-o(1)$. For this purpose we show that, for large $n$,
\begin{equation}\label{2013_08_30+10}
	\bar{N}(1-\eps)\le \E {\tilde N}'\le \bar{N}(1+2\eps)
\quad
{\text{and}}
\quad
 \pr\bigl(|{\tilde N}'-\E {\tilde N}'|\ge \eps \,\E {\tilde N}'\bigr)=o(1).
\end{equation}

The proof of (\ref{2013_08_30+10}) is routine. 
Notice that conditionally, given ${\tilde X}_i = k$, we have
${\tilde X}_i' = \sum_{j = 1}^k \xi_j$, where $\xi_1, \xi_2, \dots, \xi_k$
are independent geometric random variables with parameters 
\[
\frac m m, \ \frac {m-1} m,\ \dots, \ \frac {m-k+1} m,
\]
respectively. Since ${\tilde X}_i\le \eps \, m$, we only consider $k < \eps m$, so
\[
  \E ({\tilde X}_i' | {\tilde X}_i =k) 
= 
\frac m m + \frac m {m-1} + \dots + \frac m {m-k+1} \le \frac k {1-\eps}\le k(1+2\eps).
\]
In the last step we used $\eps\le 1/2$. We conclude that
\begin{equation}\label{eq.condexp'}
 {\tilde X}_i 
\le 
\E ({\tilde X}_i'|{\tilde X}_i) 
\le
{\tilde X}_i (1 + 2 \eps).
\end{equation}
From (\ref{eq.condexp'}) we obtain
\begin{equation}\label{2013_08_31}
 n\E{\tilde X}_1\le \E{\tilde N}'\le (1+2\eps)n\E{\tilde X}_1.
\end{equation}
Furthermore, invoking in (\ref{2013_08_31}) 
the inequalities $\E X_1-s\le \E{\tilde X}_1\le \E X_1$, where  
\begin{displaymath}
 s
=
\E X_1{\mathbb I}_{\{X_1>\eps\, m\}}
\le
(\eps\, m)^{-1}\E X_1^2 
=
(\eps\, n)^{-1}\E Y_1^2
=o(\eps),
\end{displaymath}
we obtain the first part of (\ref{2013_08_30+10}). The second part of 
(\ref{2013_08_30+10}) follows from the inequalities ${\tilde N}'\ge N(1-\eps)$ and
\begin{equation}\label{2013_08_31+1}
 Var {\tilde N}'\le 2n\E X_1^2=2m\E Y^2, 
\end{equation}
by Chebyshev's inequality. Let us show  (\ref{2013_08_31+1}). Proceeding as in the proof
of (\ref{eq.condexp'}) we evaluate the conditional variance
\[
Var({\tilde X}_i' | {\tilde X}_i = k) 
= 
\sum_{j=1}^k Var(\xi_j) 
= 
\sum_{j=0}^{k-1} \frac {j m} {(m-j)^2} \le \frac {k^2} {2 (1-\eps)^2 m}
\le 
\frac{k^2}{m},
\]
and  obtain
\[
\E Var({\tilde X}_i' | {\tilde X}_i) 
\le \frac {\E {\tilde X}_i^2} m.
\]
Furthermore, using (\ref{eq.condexp'}) we write
\[
Var(\E ({\tilde X}_i' | {\tilde X}_i)) 
\le 
\E (\E({\tilde X}_i' | {\tilde X}_i))^2 
\le 
\E {\tilde X}_i^2 (1+2\eps)^2 \le \E {\tilde X}_i^2 (1 + 8 \eps).
\]
Collecting these estimates we obtain an upper bound for the variance
\[
Var({\tilde X}_i') 
= 
\E Var({\tilde X}_i' |{\tilde X}_i) 
+ 
Var(\E({\tilde X}_i' |{\tilde X}_i))  
\le 
\E {\tilde X}_i^2 (1 + 8 \eps + m^{-1}) 
\le 
2 \E X_i^2.
\]
This bound implies (\ref{2013_08_31+1}). We have shown (\ref{2013_08_30}) for $r=2$.

Let us prove (\ref{2013_08_30}) for $r=1$. We start with an auxiliary inequality.
Given integers $x_1,\dots, x_n\ge 0$ consider a collection of $k=x_1+\dots+x_n > 0$ 
labelled 
balls, containing $x_i$ balls of label $i$, $1\le i\le n$. 
The probability of the 
event 
that a random subset of  $r$ balls contains a pair of equally labelled balls is
\begin{equation}\label{2013_08_31+3}
 \pr(L\ge 1)\le \E L=\binom{r}{2} {\binom{k}{2}}^{-1}\sum_{i}\binom{x_i}{2}
\le \Bigl(\frac{r}{k}\Bigr)^2\sum_{i}x_i^2.
\end{equation}
Here $L$ counts pairs of equally labelled balls in the  random subset.

We will show that   $\pr({\bar{\cal A}}_1)=o(1)$.
To this aim, we introduce  events
\begin{displaymath}
 {\cal A}_5=\{M({\tilde N}')\le \ln n\},
\qquad
{\cal A}_6=\{\sum_{1\le i\le n}({\tilde X}'_i)^2 \le m \ln n\},
\end{displaymath}
estimate 
\begin{displaymath}
 \pr({\bar {\cal A}}_1)
\le 
\pr({\bar {\cal A}}_1\cap {\cal A}_3 \cap {\cal A}_4 \cap{\cal A}_5\cap{\cal A}_6)
+
\pr({\bar {\cal A}}_3)
+ \pr({\bar {\cal A}}_4)
+ \pr({\bar {\cal A}}_5)
+ \pr({\bar {\cal A}}_6),
\end{displaymath}
and show that each summand on the right is $o(1)$.
For the first summand we 
estimate  using (\ref{2013_08_31+3}) 
\begin{align*}
    &\pr({\bar {\cal A}}_1\cap {\cal A}_3 \cap \ca_4 \cap {\cal A}_5 \cap{\cal A}_6)
=
\E\pr({\bar {\cal A}}_1|X_1,\dots, X_n){\mathbb I}_{{\cal A}_3\cap {\cal A}_4 \cap {\cal A}_5 \cap {\cal A}_6}
\\
 &\le 
\E \left(\frac{M({\tilde N}')^2}{({\tilde N'})^2} \sum_i({\tilde X}'_i)^2 \ii_{\ca_3 \cap \ca_4 \cap \ca_5 \cap \ca_6} | X_1, \dots, X_n \right)
\\ 
&\le
\left(\frac{\ln n}{{\bar N}^{+}}\right)^2m\ln n=O\left(\frac{\ln^3n}{n}\right).
\end{align*}

It remains to show $\pr({\bar{\cal A}}_r)=o(1)$, for $r=5,6$. We write 
$\pr({\bar {\cal A}}_5)=\pr({\bar {\cal A}}_5\cap{\cal A}_3)+o(1)$ and estimate
\begin{equation} \label{eq.Mlnn}
\pr({\bar {\cal A}}_5\cap{\cal A}_3)
\le 
\pr(M({\bar N}^{+})> \ln n)
=
\pr(\max_{j\in[m]}Z_j>\ln n)
\le 
m\pr(Z_1> \ln n)=o(1).
\end{equation}
Here $Z_j$ denotes the number of balls in the $j$th bin after  ${\bar N}^{+}$
balls have been thrown. In the second inequality we applied the union bound and used the fact that $Z_1,\dots, Z_m$ 
are identically distributed. To get the very last bound  we write 
for binomially $Bin({\bar N}^{+}, m^{-1})$ distributed
$Z_1$
and  $t=\lfloor \ln n\rfloor$,  
\begin{displaymath}
 \pr(Z_1\ge t)\le \binom{{\bar N}^{+}}{t}m^{-t}\le \left(\frac{e{\bar N}^{+}}{tm}\right)^t
 = o\left(m^{-1}\right). 
\end{displaymath}

To estimate $\pr({\bar{\cal A}}_6)$ we apply Markov's inequality, 
\begin{displaymath}
 \pr({\bar{\cal A}}_6)\le (m \ln n)^{-1}n\E({\tilde X}_1')^2
=
\ln^{-1}n (Var({\tilde X}_1')+(\E{\tilde X}_1')^2)
=
O(\ln^{-1}n).
\end{displaymath}

Finally, we prove (\ref{2013_08_30}) for $r=2$.
Notice that the coupling between $M( {\bar N}^+)$ and $M({\bar N}^-)$ is 
equivalent to the coupling provided by Lemma~\ref{lem.epsbb}. 
Choose $\eps'$ solving $N^+=(1+\eps') N^- $ and note that
$\eps' \sim 8 \eps=O(\ln^{-2}n)$.
The bound $\pr({\cal A}_2)=1-o(1)$ follows by Lemma~\ref{lem.epsbb}  
and the bound $\pr(M({\bar N}^{+}) > \ln n)=o(1)$, 
shown above.
\end{proofof}

%

\section{Algorithms for finding the largest clique}
\label{sec.algo}

Random intersection graphs provide theoretical models for real 
networks, such as 
the affilation (actor, scientific collaboration) networks.
Although the model assumptions 
about the 
distribution of the family of random sets  defining the intersection graph are 
rather stringent (independence and a particular form of the distribution), these models 
yield random graphs with clustering properties similar to those found in real networks, 
\cite{bloznelis11}. While observing a real network we may or may not have
information about the sets of attributes prescribed to vertices. Therefore it is important
to have algorithms suited to random intersection graphs that do not use any data related
to  attribute sets prescribed to vertices. In this section we consider 
two such algorithms that find cliques of order $(1+o(1))\omega(G)$ in a 
random intersection graph $G$.

The \textsc{Greedy-Clique} algorithm of \cite{jln10}
finds a clique of the optimal order $(1-o_P(1))\omega(G)$ in
a random intersection graph, 
in the case where (asymptotic) degree distribution is a power-law
with exponent $\alpha \in (1; 2)$.

\medskip

\begin{center}
\fbox{\begin{minipage}{4.5in}

\textsc{Greedy-Clique}(G):

   \smallskip

   \qquad \textit{Let $v^{(1)}, \dots, v^{(n)}$ be $V(G)$ sorted by their degrees, descending}

\qquad $M \leftarrow \emptyset$

\qquad \textbf{for} $i = 1$ to $n$ 

\qquad \qquad \textbf{if} $v^{(i)}$ \textit{is adjacent to each vertex in} $M$ \textbf{then}

\qquad \qquad \qquad $M \leftarrow M \cup \{v^{(i)}\}$

\qquad \textbf{return} $M$

\end{minipage}}
\end{center}
Here 
we assume that 
graphs are represented by 
the adjacency list data structure. 
 The implicit computational model behind our running time estimates in this section is random-access machine (RAM). 
%
\begin{prop}\label{prop.gcalg}
    Assume that conditions of Theorem \ref{thm.main} hold. 
    Suppose  that $\E Y = \Theta(1)$ and that (\ref{eq.uniform}) holds for some $\eps > 0$.
    Then on 
input $G=G(n)$ \textsc{Greedy-Clique} outputs a clique of size 
$\omega(G(n)) (1 - o_P(1))$ in time $O(n^2)$.
\end{prop}
By Lemma~\ref{lem.degrees}, the above result remains true if the conditions 
(\ref{eq.pldef2}) and $\E Y(n)=\Theta(1)$ are replaced by the conditions 
(\ref{eq.degheavytail})  and $\E D_1=\Theta(1)$. 
Proposition~\ref{prop.gcalg} is proved in a similar way as Lemma~\ref{lem.G2}, but it does 
not follow from Lemma~\ref{lem.G2}, since
\textsc{Greedy-Clique} is not allowed to know the attribute subset sizes.
The proof of Proposition~\ref{prop.gcalg} is given 
in the extended version of the paper \cite{extended}.

For random intersection graphs with square integrable degree distribution
we suggest the following simple algorithm.


\smallskip


\begin{center}
\fbox{\begin{minipage}{4.5in}

    \textsc{Mono-Clique}(G):

\smallskip

\qquad \textbf{for} $uv \in E(G)$ 

\qquad \qquad $D(uv) \leftarrow |\Gamma(u) \cap \Gamma(v)|$


\qquad \textbf{for} $uv \in E(G)$ in the decreasing order of $D(uv)$ 

\qquad \qquad $S \leftarrow \Gamma(u) \cap \Gamma(v)$

\qquad \qquad \textbf{if} $S$ \textit{is a clique} \textbf{then}  

\qquad \qquad \qquad \textbf{return} $S \cup \{u,v\}$

\qquad \textbf{return} $\{1\} \cap V(G)$
\end{minipage}}
\end{center}
Here $\Gamma(v)$ denotes the set of neighbours of $v$.

\smallskip


\begin{theorem}\label{thm.algfinite} 
Assume that  $\{G(n)\}$ is a sequence of random intersection graphs such that
 $n=O(m)$ and
$\E Y^2(n)=O(1)$.
 Let $C = C(n)$ be the clique constructed by 
 \textsc{Mono-Clique} on input $G(n)$. Then $\E \left(\omega(G(n)) - |C|\right)^2 = O(1)$. Furthermore, if there is a sequence $\{\omega_n\}$, such that $\omega_n \to \infty$ and $\omega(G(n)) \ge \omega_n$ whp, then $|C| = \omega(G(n))$ whp.
\end{theorem} 

\begin{proof} 
    Given distinct vertices $v_1, v_2, v_3, v_4 \in [n]$, let 
${\cal C}(v_1, v_2, v_3, v_4)$
    be the event that $G(n)$ contains a cycle with edges 
$\{v_1v_2, v_2v_3, v_3v_4, v_1v_4\}$ and
    $S_{v_2} \cap S_{v_4} = \emptyset$. 
Let $Z$ denote 
the number of tuples $(v_1,v_2,v_3,v_4)$
    of distinct vertices in $[n]$ such that
    ${\cal C}(v_1,v_2,v_3,v_4)$ hold. We will show below that
    \begin{equation} \label{eq.Xcyc}
       \E Z = O(1).
    \end{equation}

    Let $S \subseteq [n]$ be the (lexicographically first) largest clique of $G(n)$. Denote
$s = |S|$. If $s \le 2$ or there is a pair $\{x,y\} \subseteq S$, $x \ne y$ such that
    $G(n)[\Gamma(x) \cap \Gamma(y)]$ is a clique, 
then the algorithm returns a clique of size $s$. Otherwise, for each such pair $\{x,y\}$ 
    there are $x',y' \in   \Gamma(x) \cap \Gamma(y)$, 
$x' \ne y'$ with $x' y' \not \in E(G(n))$. That is,
    ${\cal C}(x,x',y,y')$ holds and
      $\binom s 2 \le Z$.
    Thus, if $\binom s 2 > Z$, 
the algorithm returns a clique $C$ of size  $s$. Otherwise, the algorithm may fail and return a clique $C$ of size $1$.
    In any case we have that
    \[
      s  - |C| \le \sqrt {2 Z} + 1
    \]
    and using (\ref{eq.Xcyc})
    \[
     \E (\omega(G(n)) - |C|)^2 \le \E (\sqrt{2 Z} + 1)^2 = O(1).
    \]
   Also if $\omega(G(n)) \ge \omega_n$ whp, then by  (\ref{eq.Xcyc}) and
 Markov's inequality
   \[
   \pr(|C| \ne \omega(G(n))) 
\le 
\pr(\omega(G(n)) < \omega_n) 
+ 
\pr \left(Z \ge \binom {\omega_n} 2 \right) \to 0.
   \]
   It remains to show  (\ref{eq.Xcyc}).
    What is the probability of the event ${\cal C}(1,2,3,4)$? Clearly, 
${\cal C}(1,2,3,4)$ implies at least one of the following events:
    \begin{itemize}
        \item ${\cal A}_1: $ there are distinct attributes $w_1, w_2, w_3, w_4 \in W$ 
such that $w_1 \in S_1 \cap S_2$,
$w_2 \in S_2 \cap S_3$, $w_3 \in S_3 \cap S_4$ and $w_4 \in S_1 \cap S_4$;
        \item ${\cal A}_2: $ there are distinct $w_1, w_2, w_3 \in W$, 
such that $w_1 \in S_1 \cap S_2 \cap S_3$, $w_2 \in S_3 \cap S_4$ and $w_3 \in S_1 \cap S_4$;
        \item ${\cal A}_3: $ there are distinct $w_1, w_2, w_3 \in W$, 
such that $w_1 \in S_1 \cap S_2$, $w_2 \in S_2 \cap S_3$ and $w_3 \in S_1 \cap S_3 \cap S_4$;
        \item ${\cal A}_4: $ there are distinct $w_1, w_2 \in W$, 
such that $w_1 \in S_1 \cap S_2 \cap S_3$ and $w_2 \in S_1 \cap S_3 \cap S_4$.
    \end{itemize}
    Conditioning on $X_1, X_2, X_3, X_4$ and using the union bound and independence we obtain, similarly as in Lemma~\ref{lem.PRK4}
    \begin{align*}
        &\pr ({\cal A}_1) \le (m)_4 \E \frac {(X_1)_2 (X_2)_2 (X_3)_2 (X_4)_2} {(m)_2^4} \le \frac { (\E Y^2)^4} {n^4};
        \\  &\pr({\cal A}_2) = \pr({\cal A}_3) \le (m)_3 \E \frac {(X_1)_2 X_2 (X_3)_2 (X_4)_2} { (m)_2^3 m} \le \frac {(\E Y^2)^3 (\E Y)} {m^{0.5} n^{3.5}};
    \\  &\pr({\cal A}_4) \le (m)_2 \E \frac{(X_1)_2 X_2 (X_3)_2 X_4} {(m)_2^2 m^2} \le \frac {(\E Y^2)^2 (\E Y)^2} {m n^3}.
    \end{align*}
Furthermore, by symmetry,
    \begin{displaymath}
    \E X 
\le 
(n)_4 \left( \pr({\cal A}_1) + \pr({\cal A}_2) + \pr({\cal A}_3) + \pr({\cal A}_4) \right) 
 =O(1).
    \end{displaymath}
\end{proof}

\bigskip

\begin{prop}\label{prop.monocliquetime}
    Consider a sequence of
    random intersection graphs $\{G(n)\}$ as in Theorem~\ref{thm.omega.mono}. 
	\textsc{Mono-Clique} can be implemented so that its expected running time on $G(n)$ is $O(n)$.
\end{prop}

%

\begin{proof}
  Let ${\tilde Z}$ denote the number of  4-cycles in $G(n)$, i.e., the number of
  tuples $(v_1, v_2,v_3,v_4)$ of distinct vertices in $[n]$, such that
  $v_1v_2, v_2v_3, v_3v_4, v_1 v_4 \in E(G(n))$. We will prove below that
  \begin{equation} \label{eq.Zcyc}
	  \E {\tilde Z} = O(n).
  \end{equation}
    Consider the running time of the first loop.    
    We can assume that the elements in each list in the adjacency list structure 
are sorted in increasing order (recall that vertices are 
elements of  $V=[n]$).
    Otherwise, given $G(n)$, they can be sorted using any standard sorting 
algorithm in time $O(n + \sum_{v \in [n]} D_v^2$),
    where $D_v = d_{G(n)}(v)$ is the degree of $v$ in $G(n)$.
    The intersection of two lists of lengths $k_1$ and $k_2$ can be found 
in $O(k_1 + k_2)$ time,
    so that expected total time for finding common neighbours is 
    \[ 
      O\left(n + \E \sum_{uv \in E(G(n))} (D_u + D_v)\right) 
=  
O \left(n + \E \sum_{v \in [n]} D_v^2 \right) = O(n).
    \]
    The last estimate follows by (\ref{eq.ED2}) in the proof of Lemma~\ref{lem.degvar}.

    The second loop can be implemented so that the next edge $uv$ with
    largest value of $D(uv)$ is found at each iteration 
(i.e., we do not sort the list of edges in advance). 
    In this way picking
    the next edge requires at most $c e(G(n))$ steps 
$c$ is a universal constant. We recall 
that the number of edges $uv \in E(G)$ with 
$\Gamma(u,v):=\Gamma(u)\cap \Gamma (v)\not=\emptyset$
that fail to induce a clique 
 is at most the number $Z$
 of 
cycles considered in the proof of Theorem~\ref{thm.algfinite} above. Therefore,  
 the total number of steps used in picking $D(uv)$ 
in decreasing order
 is at most 
\begin{displaymath}
 Z\,e(G(n))
=
\sum_{(i,j,k,l)}{\mathbb I}_{{\cal C}(i,j,k,l)} e(G(n)).
\end{displaymath}
Now 
\[
 e(G(n)  = \sum_{s<t:\, \{s,t\}\cap \{i,j,k,l\}=\emptyset} {\mathbb I}_{\{s\sim t\}} + \sum_{s<t:\, \{s,t\}\cap \{i,j,k,l\}\ne\emptyset} {\mathbb I}_{\{s\sim t\}}.
\]
Note, that the second sum on the right is at most $4 n$.
Also, if $\{s,t\} \cap \{i,j,k,l\} = \emptyset$, the events $s \sim t$ and ${\cal C}(i,j,k,l)$ are independent, therefore
\begin{align*}
 \E \left( {\mathbb I}_{\cc(i,j,k,l)}  \sum_{s<t:\, \{s,t\}\cap \{i,j,k,l\}=\emptyset} {\mathbb I}_{\{s\sim t\}} \right) 
        &=  \pr(\cc(i,j,k,l))  \sum_{s<t:\, \{s,t\}\cap \{i,j,k,l\}=\emptyset} \pr(s\sim t) 
	\\
 &\le \pr(\cc(i,j,k,l)) \E e(G(n)).
\end{align*}
Finally, invoking
the simple bound 
$\E e(G(n))=\tbinom{n}{2}\pr(u\sim v)=O(n)$,
and  (\ref{eq.Xcyc})  
we get
\begin{displaymath}
	\E Z\,e(G(n)) \le (\E e(G(n)) + 4 n) \sum_{(i,j,k,l)} \pr(\cc(i,j,k,l)) = (\E e(G(n)) + 4 n) \E Z = O(n).
\end{displaymath}


    Now let us estimate the time of the rest of the iteration of the second loop. 
The total expected time to find common neighbours is again $O(n)$, 
so we only consider the time spent for checking if 
$\Gamma(u,v)$ is a clique. This requires $c \, s^2_{uv}$ steps, where
 we denote $s_{uv} = |\Gamma(u,v)|$. 
Observe that  $u,v$ and $\Gamma(u,v)$
    yield at least $s_{uv} (s_{uv}-1)$ 4-cycles in $G(n)$ of the form
    $(u,x,v,y)$, $x,y \in \Gamma(u,v)$.
    Summing over all edges $uv$ and noticing that each 4-tuple corresponding to 4-cycle in $G(n)$ can be obtained at most once, we get
    \[
    {\tilde Z} 
\ge 
\sum_{uv \in E(G(n))} s_{uv}(s_{uv} -1) 
\ge   
\sum_{uv \in E(G(n))} (s_{uv}^2-1)/2.
    \]
    So using (\ref{eq.Zcyc}) and the fact that $\E e(G(n))=O(n)$ we obtain
\[
\E \sum_{uv \in E(G(n))}  s_{uv}^2
\le 2 \E \tilde{Z} + \E e(G(n)) = O(n).
\]
Finally, let us bound $\E {\tilde Z}$. Let ${\cal A}_i$, $1\le i \le 4$ 
be as in the proof of Theorem~\ref{thm.algfinite}. Let ${\cal A}_5$ 
    be the event that there is $w \in W$ such that $w \in S_1 \cap S_2 \cap S_3 \cap S_4$. Using the union bound
    \[
    \pr({\cal A}_5) \le m \E \frac {X_1 X_2 X_3 X_4} {m^4} = \frac {(\E Y)^4} {m n^2}.
    \]
    Similarly as in the proof of Theorem~\ref{thm.algfinite} 
    (we have to consider three other events similar to $A_2$ and $A_4$),
    \[
    \E {\tilde Z} \le (n)_4 \left( \pr(A_1) + 4 \pr(A_2) + 2 \pr(A_4) + \pr(A_5) \right) 
= O(n).
    \]
\end{proof}

\bigskip

Combining the next lemma with Theorem~\ref{thm.omega.mono}
we can show that \textsc{Mono-Clique} whp finds a clique of size at least $\omega'(G(n))$.

\begin{lemma} \label{lem.twogood}
    Let $\{G(n)\}$ be as in Theorem~\ref{thm.omega.mono} and let $M = M(G(n))$
    be the monochromatic clique of size $\omega'(G(n))$ generated by
    the attribute with the smallest index. 
    Then
    whp $G(n)$ has an edge $uv$ such that $\{u, v\} \cup (\Gamma(u) \cap \Gamma(v)) = M$.
\end{lemma}

The proof is given in the extended version of the paper \cite{extended}.

\section{Equivalence between set size and degree parameters}
\label{sec.tailequivalence}
Here we prove Lemmas \ref{lem.degrees} and \ref{lem.degvar}. In the proof we write
$X=X(n)$, $Y=Y(n)$, and $D_1=D_1(n)$. We denote $X_1,X_2,\dots$ the sizes of subsets 
$S_1,S_2,\dots \subseteq W$ prescribed to the vertices $1,2,\dots\in V=[n]$ of $G(n)$.

\medskip

\begin{proofof}{Lemma~\ref{lem.degrees}}
    We start by showing that if either $\E Y$ or $\E D_1$ converges
    and for some positive sequence $\{a_n\}$  converging to zero 
(we write $a=a_n$ for short), 
\begin{equation}\label{2013_09_02}
 \E Y \I_{Y >(a n)^{1/2}} \to 0
\end{equation}
    then
    \begin{equation} \label{eq.YD1}
       \E Y = (\E D_1)^{1/2} + o(1).
   \end{equation}

    We note that $\E D_1= (n-1)\pr(S_1 \cap S_2 \ne \emptyset)$. We estimate this 
probability
using the  inequalities, 
see Lemma 6 in \cite{bloznelis11},
    \begin{equation} \label{eq.estP}
\frac {X_1 X_2} {m}
\ge 
\pr(S_1 \cap S_2 \ne \emptyset | X_1, X_2)
\ge
\max\left\{0,\left(\frac {X_1 X_2} m - \frac {X_1^2 X_2^2} {m^2}\right)\right\}=:Z. 
  \end{equation}
    Notice that $\E Y = \Omega(1)$. This is clear if $\E Y \to y \in (0; \infty)$. 
    Otherwise, we have $\E D_1 \to d\in (0; \infty)$ and, by the 
first inequality of (\ref{eq.estP}),
    \[
   (n-1) \frac { (\E Y)^2} n 
\ge 
(n-1)\pr(S_1 \cap S_2 \ne \emptyset)
=
\E D_1.
    \]
Furthermore, from $\E Y=\Omega(1)$ and (\ref{2013_09_02}) we conclude that 
$\E X \I_{X\ge (am)^{1/2}}=o(\E X)$.
    Using this bound we estimate $\E Z$ from below
    \begin{align}
        \E Z &
\ge 
\E Z \I_{X_1 X_2 \le a m} 
\ge  
(1-a)m^{-1} \E X_1 X_2 \I_{X_1 X_2 \le am} 
\nonumber
        \\ & 
\ge (1-a) m^{-1} \E X_1 \E X_2 - m^{-1} \E X_1 X_2\I_{X_1 X_2 > am}, 
\label{eq.unifbound}
 \end{align}
where
    \begin{eqnarray}\label{2013-09-04++}
        \E {X_1 X_2  \I_{X_1 X_2 > am}} 
&
\le 
&
\E {X_1 X_2}\left( \I_{X_1 > (am)^{1/2}} + \I_{X_2 > (a m)^{1/2}} \right)
\\
\nonumber 
&
\le 
&
2 \E X \E X \I_{X > (a m)^{1/2}} 
\\
\nonumber
&
 = 
&
o ( (\E X) ^ 2).
    \end{eqnarray}
Hence, $\E Z\ge (1-o(1))(\E X)^2$.  Combining this inequality with (\ref{eq.estP}) we obtain
    \[
    \pr(S_1 \cap S_2 \ne \emptyset) \sim m^{-1}(\E X)^2, 
    \]
thus proving (\ref{eq.YD1}).

 It remains to prove that  (\ref{eq.pldef2})$\Leftrightarrow$(\ref{eq.degheavytail}).
Since both implications
are shown in much the same way, we 
only prove
(\ref{eq.pldef2})$\Rightarrow$(\ref{eq.degheavytail}). 
For this purpose 
we 
fix
$0<{\tilde \eps}<\min\{\eps,\, \eps_0\}$
and show that for each $0<\delta<1$ and each sequence $\{t_n\}$ with
    $n^{1/2-\tilde{\eps}} \le t_n \le n^{1/2+\tilde{\eps}}$ 
\begin{eqnarray}\label{eq.Dlower}        
&&
\liminf_n\bigl(\pr(Y_1(n) \ge t_n)/\pr (D_1 (n) \ge t_n)\bigr)
\ge 
(d^{1/2} (1+\delta))^{-\alpha},
\\
\label{eq.Dupper}
&&
\limsup_n\bigl(\pr(Y_1(n) \ge t_n)/\pr (D_1 (n) \ge t_n)\bigr) 
\le 
(d^{1/2} (1-\delta))^{-\alpha}.
\end{eqnarray}
Here the random variable $Y_1(n):=(n/m)^{1/2}X_1(n)$ has the same distribution as $Y(n)$.
We prove   (\ref{eq.Dlower}) and (\ref{eq.Dupper}) by contradiction.
  
\textit{Proof of  (\ref{eq.Dlower})}. 
    Suppose there is
    an increasing sequence $\{n_k\}$ of positive integers and
    a sequence $\{b_k\}$ with
$ n_k^{1/2 - \tilde{\eps}} \le b_k \le n_k^{1/2 +\tilde{\eps}}$ 
   such that, for some $0<\delta<1$,
\begin{equation}\label{2013_09_04}
\pr(Y_1(n_k) \ge b_k) 
< 
(d^{1/2} (1 + \delta))^{-\alpha} \pr(D_1(n_k) \ge b_k ), \quad k = 1, 2, \dots.
\end{equation}
    Define  $\{l_k\}$ by the relation 
$b_k = d^{1/2} (1 + \delta/2) l_k$, $k\ge 1$. 
Introduce events 
${\cal A}_k=\{D_1(n_k)\ge b_k\}$, ${\cal B}_k=\{Y_1(n_k)\ge l_k\}$ and write
    \begin{equation} \label{eq.Dnksplit}
     \pr( {\cal A}_k) = \pr({\cal A}_k\cap {\cal B}_k) + 
\pr({\cal A}_k\cap{\bar{\cal B}}_k).  
    \end{equation}
  In what follows we drop the subscript $k$ and write $b,l, n,m$ instead 
of $b_k,l_k,n_k,m_k$.
We note that (\ref{eq.pldef2}) together with (\ref{2013_09_04}) imply
\begin{displaymath}
\pr({\cal A}\cap {\cal B}) 
\le 
\pr({\cal B})
\sim  
d^{\alpha/2} (1 + \delta/2)^{\alpha} \pr(Y_1(n) \ge b) 
\le c_1 \pr({\cal A}),
    \end{displaymath}
where the constant $c_1 = \left((1 + \delta/2)/(1 + \delta)\right)^{\alpha} < 1$. 
Next we show that
$\pr({\cal A}\cap{\bar{\cal B}})=O(n^{-10})$ thus obtaining a contradiction to 
(\ref{2013_09_04}),
(\ref{eq.Dnksplit}).    

Denote $x = \left \lfloor (m/n)^{1/2} l \right  \rfloor$. Conditionally, given
the event ${\cal C}=\{X_1(n)=x\}$, the random variable
$D_1(n)$ has binomial distribution $Bin(n-1, p)$ with success probability
$p=\pr(S_1\cap S_2\not=\emptyset\bigr|\, |S_1|=x)$ satisfying $p\sim d^{1/2}l/n$. Indeed, 
the first inequality of
 (\ref{eq.estP}) implies
\begin{equation}\nonumber
 p
\le 
\frac{x\E X_2}{m}
=
\frac{x(m/n)^{1/2}\E Y}{m}
\sim
d^{1/2}\frac{l}{n}.
\end{equation}
Here we used $\E Y\to d^{1/2}>0$. 
The second inequality of  (\ref{eq.estP}) implies, see  (\ref{eq.unifbound}),
\begin{equation}\nonumber
 p\ge \frac{1-a}{m}x\E X_2{\mathbb I}_{\{xX_2<am\}}
=
\frac{1-a}{m}x(\E X_2-r)
\sim
\frac{x\E X_2}{m}.
\end{equation}
Here $r=\E X_2{\mathbb I}_{\{xX_2\ge am\}}=o(\E X_2)$, for $a=a(n_k)=\ln^{-1}n_k$, cf.
(\ref{2013-09-04++}).

Next,  since
$b\sim (1+\delta/2)np$ and $np\sim d^{1/2}l=\Omega(n^{1/2-{\tilde \eps}})$
we obtain, 
by Chernoff's inequality, $\pr({\cal A}|{\cal C})
=O(n^{-10})$. Now, using the  inequality
$\pr({\cal A}|Y_1(n)=y)\le \pr({\cal A}|{\cal C})$, for $y\le l$, we obtain 
\begin{equation}\label{2013-09-04+4}
 \pr({\cal A}\cap {\bar{\cal B}})=\E \pr({\cal A}|Y_1(n)){\mathbb I}_{\{Y(n)\le l\}}
\le \pr({\cal A}|{\cal C})=O(n^{-10}).
\end{equation}

\textit{Proof of  (\ref{eq.Dupper})}.
Suppose there is
    an increasing sequence $\{n_k\}$ of positive integers and
    a sequence $\{b_k\}$ with
$ n_k^{1/2 - \tilde{\eps}} \le b_k \le n_k^{1/2 +\tilde{\eps}}$ 
   such that, for some $0<\delta<1$,
\begin{equation}\label{2013_09_04_}
\pr(Y_1(n_k) \ge b_k) 
>
(d^{1/2} (1 - \delta))^{-\alpha} \pr(D_1(n_k) \ge b_k ), \quad k = 1, 2, \dots.
\end{equation}
    Define  $\{l_k\}$ by the relation 
$b_k = d^{1/2} (1 - \delta/2) l_k$, $k\ge 1$. 
We write
  \begin{equation}\label{eq.Duppereq}
   \pr(D_1(n_k) \ge b_k) =\pr(Y_1(n_k) \ge l_k) \pr(D_1(n_k) \ge b_k | Y_1(n_k) \ge l_k).
  \end{equation}
  We note that, by (\ref{eq.pldef2}) and (\ref{2013_09_04_}), the first term on 
the right is at least 
  $(c_2 + o(1))$ $ \pr (D_1(n_k)$ $ \ge b_k)$
  where the constant $c_2 = \left( (1 - \delta/2)/(1-\delta) \right)^\alpha > 1$.
  Finally, we 
obtain a contradiction, by showing that the second term of (\ref{eq.Duppereq}) 
is $1 - O(n^{-10})$. Here  we proceed as in (\ref{2013-09-04+4}) above. We 
write
\begin{displaymath}
\pr(D_1(n_k)<b_k|Y_1(n_k)\ge l_k)\le \pr(D_1(n_k)<b_k|{\cal C}) 
\end{displaymath}
and show that  binomial probability on the right-hand side is $O(n^{-10})$ using
 Chernoff's inequality.

\end{proofof}


\medskip

\begin{proofof} {Lemma~\ref{lem.degvar}} 
    The identity (\ref{eq.ED12}) follows from (\ref{eq.YD1}) since
    \[
      \E Y \I_{Y > \eps_n n^{1/2}} \le (\E Y^2 \I_{Y > \eps_n n^{1/2}})^{1/2} \to 0.
      \] 
    Let us show (\ref{eq.VarD}). 
Denote $N$ the number of $2$-stars in $G=G(n)$ centered at vertex $1\in V=[n]$. Introduce 
the events ${\cal A}_{ij}=\{i\sim j\}$, $i,j\in V$.  
Write, for short, ${\cal A}={\cal A}_{12}\cap {\cal A}_{13}$. Let ${\tilde \pr}$ denote the
conditional probability given the sizes $X_1,X_2, X_3$ of the random subsets prescribed 
to vertices $1,2,3\in V$.
We remark that (\ref{eq.VarD}) follows from (\ref{eq.ED12}) combined with the 
simple identities
\begin{displaymath}
\E D_1 (D_1 - 1) = 2 \E N=(n-1)(n-2)\pr({\cal A}),
\end{displaymath}
and the inequalities 
\begin{equation}\label{2013_09_06}
(\E Y)^2\E Y^2
\ge
n^2\pr({\cal A})
\ge
 (1-o(1))(\E Y)^2\E Y^2.
\end{equation}
Let us prove (\ref{2013_09_06}). For this purpose we write 
(using the conditional independence of 
events ${\cal A}_{12}$ and ${\cal A}_{13}$, given $X_1,X_2,X_3$)
\begin{displaymath}
\pr({\cal A})
=
\E{\tilde\pr}({\cal A})
=
\E{\tilde \pr}({\cal A}_{12}){\tilde \pr}({\cal A}_{13}) 
\end{displaymath}
and evaluate conditional probabilities ${\tilde \pr}({\cal A}_{ij})$ using
(\ref{eq.estP}). 
From the first inequality of (\ref{eq.estP}) we obtain the first inequality of 
(\ref{2013_09_06})
\begin{displaymath}
\pr({\cal A})
=
\E{\tilde \pr}({\cal A}_{12}){\tilde \pr}({\cal A}_{13})
\le
\E(X_1^2X_2X_3)/m^2
=
(\E Y)^2\E Y^2/n^2.
\end{displaymath}
Thus, even without the assumption (\ref{eq.uni2}) (we use this fact 
this in the proof of Proposition~\ref{prop.monocliquetime}),
we have
\begin{equation} \label{eq.ED2}
    \E D_1 \le \E Y \quad \mbox{and} \quad \E D_1 (D_1-1) \le \E Y^2 \E Y.
\end{equation}

To show the second inequality of (\ref{2013_09_06}) we apply the second inequality 
of (\ref{eq.estP}) and use  truncation. We denote 
${\mathbb I}_{i}={\mathbb I}_{\{X_i\le\eps_nm^{1/2}\}}$,  
${\bar {\mathbb I}}_{i}=1-{\mathbb I}_{i}$ and write, cf. (\ref{eq.unifbound}),
\begin{eqnarray}\nonumber
 \pr(A)
&
\ge
&
\E{\tilde \pr}({\cal A}){\mathbb I}_1{\mathbb I}_2{\mathbb I}_3
\ge
(1-\eps_n^2)^2\E (X_1^2X_2X_3/m^2){\mathbb I}_1{\mathbb I}_2{\mathbb I}_3
\\
\nonumber
&
\ge 
&
(1-\eps_n^2)^2\E (X_1^2X_2X_3/m^2)
(1-{\bar{\mathbb I}}_1-{\bar{\mathbb I}}_2-{\bar{\mathbb I}}_3)
\\
\nonumber
&
=
&
(1-o(1))(\E Y)^2\E Y^2/n^2.
\end{eqnarray}
In the last step we used the fact that $\E Y^2\ge (\E Y)^2=\Omega(1)$ and the bounds
\begin{eqnarray}\nonumber
&&
\E X_1^2 {\bar{\mathbb I}}_1
= 
(m/n)\, \E Y^2 {\mathbb I}_{\{Y > \eps_n n^{1/2}\}} 
= 
o(\E X^2),
\\
\nonumber
&&
\E X_j {\bar{\mathbb I}}_j  
= 
(m/n)^{1/2} \E Y {\mathbb I}_{\{Y > \eps_n n^{1/2}\}} 
= 
o(\E X),
\qquad
j=2,3.
\end{eqnarray}
\end{proofof}

\section{Concluding remarks}

In this work we determined the order of the clique number in $G(n,m,P)$ for
a wide range of $m = m(n)$ and $P = P(n)$. We saw that in sparse power-law random intersection 
graphs with unbounded degree variance, the clustering property of $G(n,m,P)$ has little influence
in the formation of the maximum clique. This suggests that simpler models, 
such as the one in \cite{jln10}, may be preferable 
in the case of 
very heavy degree tails.
However, when the degree variance is bounded, most random graph models,
including the Erd\H{o}s-R\'enyi graph and the model of \cite{jln10} have 
only bounded size cliques whp.
In contrast, we showed that in random intersection graphs the clique number can still diverge slowly. 

We have a kind of ``phase transition'' 
as the tail index $\alpha$ for
 the random subset size (degree) varies, see (\ref{eq.pldef2}).
Assume, for example that $m = \Theta(n)$.
When $\alpha < 2$, the random graph $G(n,m,P)$ whp contains
cliques of only logarithmic size. When $\alpha > 2$, it whp contains a `giant' clique of polynomial size.
But what happens when (\ref{eq.pldef2}) is satisfied with $\alpha = 2$ but the degree variance is unbounded?

We proposed a surprisingly simple algorithm for finding (almost) the largest clique in
sparse random intersection graphs with finite degree variance. The performance of both \textsc{Greedy-Clique}
and \textsc{Mono-Clique} algorithms can be of further interest, since these algorithms do not use the
possibly hidden random subset structure. How well would they perform on arbitrary sparse 
empirical networks? Can we suspect a hidden intersecting sets structure for networks where the \textsc{Mono-Clique} algorithm performs well?

Another direction of possible future research would be to determine the asymptotic clique number in dense random intersection graphs (alternatively, the order of the largest intersecting set in dense random hypergraphs).
For example, even in the random uniform hypergraph case where $m = \Theta(n)$ and the random subset size $X(n) = \Omega(n^{1/2})$ is deterministic, exact asymptotics of the clique number remain open. 
\end{document}